\documentclass[11pt]{article}

\usepackage{amsmath,amssymb,amsthm,graphicx,subfigure,float,url,mathrsfs}

\usepackage{pdfsync}

\usepackage[usenames]{color}
\usepackage{fullpage}

\newcommand{\R} {\mathbb{R}}

\newcommand{\Vf}{\mathcal{V}_{f}}   
\newcommand{\Vperp}{\mathcal{V}_{\perp}} 
\newcommand{\RN}{(\mathbb{R}^{d})^{N}} 

\renewcommand{\geq}{\geqslant}
\renewcommand{\leq}{\leqslant}

\newtheorem{theorem}{Theorem}  
\newtheorem{proposition}{Proposition}
\newtheorem{corollary}{Corollary}
\newtheorem{definition}{Definition}
\newtheorem{lemma}{Lemma}
\newtheorem{example}{Example}
\theoremstyle{definition}\newtheorem{remark}{Remark}

\begin{document}

\title{Sparse Stabilization and Control of Alignment Models}
\author{Marco Caponigro\footnote{Conservatoire National des Arts et M\'etiers, \'Equipe M2N,
292 rue Saint-Martin, 75003, Paris, France. ({\tt marco.caponigro@cnam.fr})},
Massimo Fornasier\footnote{Technische Universit\"at M\"unchen, Facult\"at Mathematik, Boltzmannstrasse 3
D-85748, Garching bei M\"unchen, Germany  ({\tt massimo.fornasier@ma.tum.de}).}, 
Benedetto Piccoli\footnote{Rutgers University, Department of Mathematics, Business $\&$ Science Building Room 325 Camden, NJ 08102, USA ({\tt piccoli@camden.rutgers.edu}).},
Emmanuel Tr\'elat\footnote{Universit\'e Pierre et Marie Curie (Univ. Paris 6) and Institut Universitaire de France,
CNRS UMR 7598, Laboratoire Jacques-Louis Lions, F-75005, Paris, France ({\tt emmanuel.trelat@upmc.fr}).}}
\maketitle

\begin{abstract}
Starting with the seminal papers of Reynolds (1987), Vicsek et. al. (1995), Cucker--Smale (2007)
there has been a flood of recent works on models of self-alignment and consensus dynamics.
Self-organization has been so far the main driving concept of this research direction.
However, the evidence that in practice self-organization does not necessarily occur
(for instance, the achievement of unanimous consensus in government decisions) leads
to the natural question of whether it is possible to {\it externally} influence the dynamics in order to
promote the formation of certain desired patterns.
Once this fundamental question is posed, one is also faced with the issue of defining
the best way of obtaining the result, seeking for the most ``economical'' way to achieve a certain
outcome. Our paper precisely addressed the issue of finding the sparsest control strategy in
order to lead us optimally towards a given outcome, in this case the
achievement of a state where the group will be able by self-organization to reach
an alignment consensus.  As a consequence we provide a mathematical justification to the general principle according to which ``sparse is better'' in the sense that a policy maker, who is not allowed to predict future developments, should always consider more favorable to intervene with stronger action on the fewest possible instantaneous optimal leaders rather than trying to control more agents with minor strength in order to achieve group consensus. 
We then establish local and global sparse controllability properties to consensus. 
 Finally, we analyze the sparsity of solutions of the finite time optimal control problem where the minimization criterion is a combination of the distance from consensus and of the $\ell_1$-norm of the control. Such an optimization  models the situation where the policy maker is actually allowed to observe future developments. We show that the lacunarity of sparsity is related to the codimension of certain manifolds in the space of cotangent vectors.
\end{abstract}



\tableofcontents

\section{Introduction}
\subsection{Self-organization Vs organization via intervention}

In recent years there has been a very fast growing interest in defining and analyzing mathematical models of multiple interacting agents in social dynamics.
Usually individual based models, described by suitable dynamical systems, constitute the basis for developing  continuum descriptions of the agent distribution, governed by suitable partial differential equations.
There are many inspiring applications, such as animal behavior, where the coordinated movement of groups, such as birds (starlings, geese, etc.),
fishes (tuna, capelin, etc.), insects (locusts, ants, bees, termites, etc.) or certain mammals (wildebeasts, sheep, etc.) is considered, see, e.g., \cite{BCCCCGLOPPVZ09, CDFSTB03, CouFra02, CKFL05, Niw94, PE99, ParVisGru02, Rom96,TonTu95,vicsek} or  the review chapter \cite{cafotove10}, and the numerous references therein. Models in microbiology, such as the Patlak-Keller-Segel model \cite{kese70,pa53}, describing the chemotactical aggregation of cells and multicellular micro-organisms, inspired a very rich mathematical
literature \cite{ho03,ho04,be07}, see also the very recent work \cite{blcaca12} and references therein. Human motion, including pedestrian and crowd modeling \cite{crpito10,crpito11,realistic-crowd, traffic-instabilities}, for instance in evacuation process simulations, has been a matter of intensive research, connecting
also with new developments such as mean field games, see \cite{lawo11} and the overview in its Section 2. Certain aspects of human social behavior, as in language evolution \cite{CS,CucSmaZho04,KeMinAuWan02} or even criminal activities \cite{ MR2438215}, are also subject of intensive study by means of dynamical systems and kinetic models. Moreover, relevant results appeared in the economical realm with the theoretical derivation of wealth distributions \cite{dumato08} and, again in connection with game theory, the description of formation of volatility in financial markets \cite{lali07}. 
Beside applications where biological agents, animals and micro-(multi)cellular organisms, or humans are involved, also more abstract modeling of interacting automatic units, 
for instance simple robots, are of high practical interest \cite{ChuHuaDorBer07,  MR2000132, SugSan97,LeoFio01,PerGomElo09,SPL07a}. \\
One of the leading concepts behind the modeling of multiagent interaction in the past few years has been \textit{self-organization} \cite{CDFSTB03, Niw94, PE99,ParVisGru02,TonTu95}, which, from a mathematical point of
view, can be described as the formation of patterns, to which the systems tend naturally to be attracted. The fascinating mechanism to be revealed by such a modeling is
how to connect the microscopical and usually binary rules or social forces of interaction between individuals  with the eventual global behavior or group pattern,
forming as a superposition in time of the different microscopical effects.  Hence, one of the interesting issues of such socio-dynamical models is the global convergence to stable patterns or, as more often and more realistically, the instabilities and local convergence \cite{be07}.
\\
While the description of pattern formation can explain some relevant real-life behaviors, it is also of paramount interest how one may enforce and stabilize pattern formation in those situations where global and stable
convergence cannot be ensured, especially in presence of noise \cite{YEECBKMS09}, or, vice versa, how one can avoid certain rare and dangerous patterns to form, despite that the system may suddenly tend stably to them. The latter situations may refer, for instance,
to the so-called ``black swans'', usually referred to critical (financial or social) events \cite{beheto12,ta10}. In all these situations where the independent behavior of the system, despite its natural tendencies, does not realize the desired result, the active intervention of an external policy maker is essential. This naturally raises the question of which optimal policy should be considered.  \\
In information theory, the best possible way of representing data is usually the most economical for reliably or robustly storing and communicating data. One of the modern ways of describing economical description of data is their {\it sparse} representation with respect to an adapted {\it dictionary} \cite[Chapter 1]{ma09}. In this paper we shall translate these concepts to realize {\it best policies} in stabilization and control of dynamical systems modeling multiagent interactions.
Beside stabilization strategies in collective behavior already considered in the recent literature, see e.g. \cite{ammemema09,SPL07a}, the conceptually closest work to our approach is perhaps the seminal paper \cite{LeoFio01}, where externally driven ``virtual leaders'' are inserted in a collective motion dynamics in order to enforce a certain behavior. Nevertheless our modeling still differs significantly from this mentioned literature, because we allow us directly, externally, and instantaneously to control the individuals of the group, with no need of introducing predetermined virtual leaders, and we shall specifically seek for the most economical (sparsest) way of leading the group to a certain behavior.
In particular, we will mathematically model {\it sparse controls}, designed to promote the minimal amount of intervention of an external policy maker, in order to enforce  nevertheless the formation of certain interesting patterns. In other words we shall activate in time the minimal amount of parameters, potentially limited to certain admissible classes, which can provide a certain verifiable outcome of our system. The relationship between parameter choices and result will be usually highly nonlinear, especially for several known dynamical systems, modeling social dynamics. Were this relationship linear instead, then a rather well-established theory would predict how many degrees of freedom are minimally necessary to achieve the expected outcome, and, depending on certain spectral properties of the linear model, allows also for efficient algorithms to compute them.  This theory is known in mathematical signal processing under the name of {\it compressed sensing}, see the seminal work \cite{carota06-1} and \cite{do06-2}, see also the review chapter \cite{fora10}.
The major contribution of these papers was
to realize that one can combine the power of convex optimization, in particular $\ell_1$-norm minimization, and spectral properties of random linear models in order to show optimal results on the ability of $\ell_1$-norm minimization of recovering robustly sparsest solutions.  
Borrowing a leaf from compressed sensing, we will model sparse stabilization and control strategies  by penalizing
the class of vector valued controls $u=(u_1,\dots, u_N) \in (\mathbb R^d)^N$ by means of a mixed $\ell_1^N-\ell_2^d$-norm, i.e.,
$$
\sum_{i=1}^N \| u_i \|,
$$
where here $\| \cdot \|$ is the $\ell_2^d$-Euclidean norm on $\mathbb R^d$.
This mixed norm has been used for instance in \cite{fora08} as a \textit{joint sparsity} constraint
and it has the effect of optimally sparsifying multivariate vectors in compressed sensing problems \cite{elra10}.
The use of (scalar) $\ell_1$-norms to penalize controls dates back to the 60's with the models of linear fuel consumption \cite{crlo65}. 
More recent work in dynamical systems \cite{voma06} resumes again $\ell_1$-minimization emphasizing its sparsifying power. 
Also in optimal control with partial differential equation constraints it became rather popular to use $L_1$-minimization to enforce sparsity of controls, for instance
in the modeling of optimal placing of actuators or sensors \cite{caclku12,clku11,clku12,hestwa12,pive12,st09,wawa11}. \\
Differently from this previously mentioned work, we will investigate in this paper optimally sparse stabilization
and control to enforce pattern formation or, more precisely, convergence to attractors in dynamical systems modeling multiagent interaction.
A simple, but still rather interesting and prototypical situation  is given by the individual based
particle system we are considering here as a particular case
\begin{equation}
\left\{
\begin{split}
\dot{x}_i&=v_i\\
\dot{v}_i&=\frac{1}{N}\sum_{j=1}^N \frac{v_j-v_i}{(1+\Vert x_j-x_i\Vert^2)^\beta}
\end{split}
\right.
\end{equation}
for $i=1,\ldots,N$, where $\beta>0$ and  $x_i\in\R^d$, $v_i\in\R^d$ are the {\it state and consensus parameters} respectively.  Here one may want to imagine that the $v_i$'s actually represent abstract quantities such as words of a communication language, opinions, invested capitals, preferences, but also more classical physical quantities such as the velocities in a collective motion dynamics. 
 This model describes the {\it emerging of consensus} in a group of $N$ interacting
agents described by $2d$ degrees of freedom each, trying to align the consensus parameters $v_{i}$ (also in terms of abstract consensus) with their social neighbors. One
of the motivations of this model proposed by Cucker and Smale was in fact to describe the formation
and evolution of languages \cite[Section 6]{CS}, although, due to its simplicity, it has been
eventually related mainly to the description of the emergence of {\it consensus} in a group of moving agents, for instance flocking in a swarm of birds \cite{cusm07}.
One of the interesting features of this simple system is its rather complete analytical description in terms
of its ability of convergence to attractors according to the parameter $\beta>0$ which is ruling the {\it communication rate} between
far distant agents. For $\beta \leq \frac{1}{2}$, corresponding to a still rather strong long - social - distance interaction, for every initial condition the system will converge to a consensus
pattern, characterized by the fact that all the parameters $v_i(t)$'s will tend for $t \to +\infty$ to the mean 
$\bar v = \frac{1}{N} \sum_{i=1}^N v_i(t)$ which is actually an invariant of the dynamics. 
For $\beta > \frac{1}{2}$, the emergence of consensus happens only under certain configurations of state variables and consensus parameters, i.e., when the group is sufficiently close to its state center of mass or when the consensus parameters are sufficiently
close to their mean. Nothing instead can be said a priori when at the same time one has
$\beta > \frac{1}{2}$ and the mentioned conditions on the initial data are not fulfilled. Actually one can easily construct counterexamples to formation of
consensus, see our Example \ref{ex:nonflock} below.
In this situation, it is interesting to consider  external control strategies which will facilitate the formation of consensus, which is precisely the scope of this work.

\subsection{The general Cucker--Smale model and introduction to its control}\label{sec:CSmodel}
Let us introduce the more general Cucker--Smale model under consideration in this article.

\paragraph{The model.}
We consider a system of $N$ interacting agents. The state of each agent is described by a pair $(x_{i}, v_{i})$ of vectors of the Euclidean space $\R^{d}$, where $x_{i}$ represents the \emph{main state} of an agent and the $v_{i}$ its \emph{consensus parameter}. 
 The main state of the group of $N$ agents is given by the $N$-uple $x=(x_{1}, \ldots, x_{N})$. Similarly for the consensus parameters $v=(v_{1},\ldots,v_{N})$. The space of main states and the space of consensus parameters is $\RN$ for both, the product $N$-times of the  Euclidean space $\R^{d}$ endowed with the induced inner product structure. 

The time evolution of the state $(x_i,v_i)$ of the $i^\textrm{th}$ agent is governed by the equations
\begin{equation}\label{system}
\left\{
\begin{split}
\dot{x}_i(t)&=v_i(t), \\
\dot{v}_i(t)&=\frac{1} {N}\sum_{j=1}^N a(\|x_{j}(t) - x_{i}(t)\|)(v_j(t)-v_i(t)),\end{split}
\right. 
\end{equation}
for every $i=1,\ldots,N$, where $a \in C^{1}([0,+\infty))$ is a {\it nonincreasing positive function}.
Here, $\Vert\cdot\Vert$ denotes again the $\ell_2^d$-Euclidean norm in $\R^d$. 
The meaning of the second equation is that each agent adjusts its consensus parameter with those of other agents in relation with a weighted average of the differences. The influence of the $j^\textrm{th}$ agent on the dynamics of the $i^\textrm{th}$ is a function of the (social) distance of the two agents.  Note that the mean consensus parameter $ \bar v = { \frac{1}{N} } \sum_{i=1}^{N} v_{i}(t)$ is an invariant of the dynamics, hence it is constant in time.

As mentioned previously, an example of a system of the form~\eqref{system} is the influential model of Cucker and Smale~\cite{CS} in which the function $a$ is of the form
\begin{equation}\label{eq:CSpotential}
a(\|x_{j} - x_{i}\|) = \frac{K}{(\sigma^{2}+\|x_{i} - x_{j}\|^{2})^{\beta}},
\end{equation}
where $K>0$, $\sigma > 0$, and $\beta \geq 0$ are constants accounting for the social properties of the group of agents.

In matrix notation, System~\eqref{system} can be written as
\begin{equation}\label{systemmatrix}
\left\{
\begin{split}
\dot x &= v \\
\dot v &= - L_{x} v,
\end{split}
\right.
\end{equation} 
where $L_{x}$ is the Laplacian%
\footnote{Given a real $N\times N$ matrix $A = (a_{ij})_{i,j}$ and $v\in \RN$ we denote by $Av$ the action of $A$ on $\RN$  by mapping $v$ to $(a_{i1}v_{1} + \cdots + a_{iN}v_{N})_{i=1,\ldots,N}$. Given a nonnegative symmetric $N \times N$ matrix  $A = (a_{ij})_{i,j}$, the \textit{Laplacian} $L$ of $A$ is defined by $L = D - A$, with $D = \mathrm{diag} (d_{1}, \ldots, d_{N})$ and $d_{k} = \sum_{j=1}^{N} a_{kj}$.}
of the $N\times N$ matrix $\left( a(\Vert x_j-x_i\Vert)/N\right)_{i,j=1}^{N}$ and depends on $x$. The Laplacian $L_{x}$ is an $N \times N$ matrix acting on $\RN$, and verifies $L_{x}(v, \ldots, v) = 0$ for every $v \in \R^{d}$. Notice that the operator $L_{x}$ always is positive semidefinite.

\paragraph{Consensus.}
For every $v \in \RN$, we define the mean vector $\bar v = \frac{1}{N} \sum_{i=1}^{N} v_{i}$ and the symmetric bilinear form $B$ on $\RN\times\RN$ by
$$
B(u,v) = \frac{1}{2N^{2}}\sum_{i,j=1}^N \langle u_{i} - u_{j}, v_{i} - v_{j} \rangle = \frac{1}{N}  \sum_{i=1}^{N}\langle u_{i}, v_{i}\rangle - \langle \bar u, \bar v  \rangle,
$$
where $\langle \cdot, \cdot \rangle$ denotes the scalar product in $\R^{d}$. 
We set
\begin{equation}\label{defVF}
\Vf =\{(v_1,\ldots,v_N) \in \RN\ \vert\ v_1= \dots =v_N\},
\end{equation}
\begin{equation}\label{defVperp}
\Vperp = \{(v_{1},\ldots,v_{N}) \in \RN\ \vert\  \sum_{i=1}^{N}v_{i}=0\}.
\end{equation}
These are two orthogonal subspaces of $\RN$.
Every $v \in \RN$ can be written as 
$v = v_{f} + v_{\perp}$ with $v_{f}  = (\bar v, \ldots, \bar v) \in \Vf$
and $v_{\perp} \in \Vperp$.

Note that $B$ restricted to $\Vperp \times \Vperp$ coincides, up to the factor $1/N$, with  the scalar product on $\RN$.
Moreover $B(u,v) = B(u_{\perp}, v) =  B(u,v_{\perp}) = B(u_{\perp},v_{\perp})$. Indeed $B(u,v_{f}) =0 = B(u_{f},v)$ for every $u,v \in \RN$.

Given a solution $(x(t),v(t))$ of~\eqref{system} we define the quantities
$$
X(t) := B(x(t),x(t)) = \frac{1}{2N^{2}} \sum_{i,j=1}^N \|x_{i}(t) - x_{j}(t)\|^{2},
$$
and
$$
V(t) := B(v(t),v(t)) = \frac{1}{2N^{2}} \sum_{i,j=1}^N \|v_{i}(t) - v_{j}(t)\|^{2} = \frac{1}{N} \sum_{i=1}^{N}\|v(t)_{\perp_{i}}\|^{2}.
$$
Consensus is a state in which  all agents have the same consensus parameter.

\begin{definition}[Consensus point]\label{def:consensuspoint}
A steady configuration of System \eqref{system} $(x,v) \in \RN\times \Vf$ is called a \emph{consensus point} in the sense that the dynamics originating from $(x,v)$ is simply given by rigid translation $x(t) = x + t \bar v$. We call $\RN\times \Vf$ the \emph{consensus manifold}.

\end{definition}

\begin{definition}[Consensus]\label{def:flocking}
We say that a solution $(x(t), v(t))$ of System \eqref{system} tends to \emph{consensus} if the consensus parameter vectors tend to the mean  $\bar v =\frac{1}{N} \sum_{i}v_{i}$, namely if $\lim_{t \to \infty} |v_{i}(t) - \bar v|=0$ for every $i=1,\ldots,N$.  
\end{definition}

\begin{remark}\label{rem:1}
Because of uniqueness, a solution of \eqref{system} cannot reach consensus within finite time, unless the initial datum is already a consensus point.  The consensus manifold is invariant for \eqref{system}.
\end{remark}

\begin{remark}\label{rk:flockequiv}
The following definitions of \emph{{consensus}} are equivalent:
\begin{itemize}
\item[$(i)$]  $\lim_{t \to \infty} v_{i}(t) = \bar v$ for every $i=1,\ldots,N$;
\item[$(ii)$] $\lim_{t \to \infty} v_{\perp_{i}}(t) = 0$ for every $i=1,\ldots,N$;
\item[$(iii)$] $\lim_{t \to \infty} V(t) =0$.
\end{itemize}
\end{remark}

The following lemma, whose proof is given in the Appendix, shows that actually $V(t)$ is a Lyapunov functional for the dynamics of \eqref{system}.

\begin{lemma} \label{lem:app1}
For every $t \geq 0$
$$
\frac{d}{dt}V(t) \leq - 2 a \left(\sqrt{2NX(t)}\right) V(t).
$$
In particular if $\sup_{t\geq 0}X(t) \leq \bar X$ then $\lim_{t \to \infty} V(t) =0$.
\end{lemma}

For multi-agent systems of the form~\eqref{system} sufficient conditions for consensus emergence are a particular case of the main result of~\cite{HaHaKim} and are summarized in the following proposition, whose proof is recalled in the Appendix, for self-containedness and reader's convenience.

\begin{proposition}\label{prop:condition-for-consensus}
Let $(x_{0},v_{0}) \in \RN \times \RN$ be  such that $X_{0}=B(x_{0},x_{0})$ and $V_{0} = B(v_{0},v_{0})$ satisfy
\begin{equation}\label{eq:CScondition}
\int_{\sqrt{X_{0}}}^{\infty} a (\sqrt{2N}r) dr \geq  \sqrt{V_{0}}\,.
\end{equation}
Then the solution of \eqref{system} with initial data $(x_{0},v_{0})$ tends to consensus.
\end{proposition}

The meaning of \eqref{eq:CScondition} is that as soon as $V_{0}$ and $X_0$ are sufficiently small then the system tends to consensus. In other words if the disagreement of the consensus parameters is sufficiently small and the initial main states are sufficiently close then the system tends to consensus.

\begin{definition}[Consensus Region]
We call the set of $(x,v) \in \RN\times \RN$ satisfying~\eqref{eq:CScondition}  the \emph{consensus region}.
\end{definition}

The consensus region represents an estimate on the basin of attraction of the consensus manifold. 
This estimate is, in some simple case, sharp as showed in Example~\ref{ex:nonflock} below.

Although consensus forms a rigidly translating stable pattern for the system and represents in some sense a ``convenient'' choice for the group, there are initial conditions for which the system does not tend to consensus, as the following example shows.

\begin{example}[Cucker--Smale system: two agents on the line]\label{ex:nonflock}
Consider the Cucker--Smale system~\eqref{system}-\eqref{eq:CSpotential} in the case of two agents moving on $\R$ with position and velocity at time $t$, $(x_{1}(t),v_{1}(t))$ and $(x_{2}(t),v_{2}(t))$. Assume for simplicity that $\beta =1, K=2$, and $\sigma=1$. Let $x(t) = x_{1}(t) - x_{2} (t)$ be the relative main state and $v(t) = v_{1} (t) - v_{2} (t)$ be the relative consensus parameter. Equation~\eqref{system}, then reads
$$
\left\{
\begin{split}
\dot x &= v\\
\dot v &= -\frac{v}{1+x^{2}}
\end{split}
\right.
$$
with initial conditions $x(0)=x_{0}$ and $v(0)=v_{0} > 0$. 
The solution of this system can be found by direct integration, {as from $\dot v = - \dot x / (1+x^{2})$ we have}
$$
v(t) -  v_{0} = -\arctan{x(t)} + \arctan{x_{0}}.
$$
If the initial conditions satisfy $|\arctan{x_{0}} + v_{0}|  < \pi/2$ then, as a consequence of Remark \ref{rem:1}, the relative main state $|x(t)|$ is bounded uniformly by $\tan{(|\arctan{x_{0}} + v_{0} |)}$, otherwise we would have $v(t^*) =0$ for a finite $t^*$. The boundedness of $x(t)$ fulfills the sufficient condition on the states in  Lemma \ref{lem:app1} for {consensus}.  If $ |\arctan{x_{0}} + v_{0}|  = \pi/2$  then the system tends to consensus as well, since $v(t) = \pm \pi/2 - \arctan{x(t)}$, depending on whether $\pm v_0 >0$ respectively: if $x(t)$ were unbounded then $\lim_{t \to \infty} x(t)= \pm \infty$, respectively, and necessarily we converged to consensus. If $x(t)$ were bounded then again by Lemma \ref{lem:app1} we would converge to consensus. \\
On the other hand, whenever $|\arctan{x_{0}} + v_{0}| > \pi/2 $,  which implies $|\arctan{x_{0}} + v_{0}| \geq  \pi/2 +\varepsilon$ for some $\varepsilon>0$, the  consensus parameter $v(t)$ remains bounded away from $0$ for every time, since
$$
|v(t)|= |-\arctan{x(t)} + \arctan{x_{0}} + v_{0}| \geq | -\arctan{x(t)} + \pi/2 +\varepsilon| > \varepsilon, 
$$
for every $t>0$.
In other words, the system does not tend to {consensus}.


\end{example}

Let us mention that by now there are several extensions of Cucker--Smale models of consensus towards addressing the presence of noise, collision-avoidance forces, non-symmetric communication, informed agents, and tolerance to faults. For a state of the art review on the current developments on such generalization we refer to \cite[Section 4.4.1]{viza12}. We mention in particular the recent work of Cucker and Dong \cite{CuckerDong11}, which modifies the original model by consider cohesion and avoidance forces. We shall return to this model in Section \ref{sec:ext} where we deal with extensions of our work.

\paragraph{Control of the Cucker--Smale model.}

When the consensus in a group of agents is not achieved by self-organization of the group, as in Example \ref{ex:nonflock} in case of $|\arctan{x_{0}} + v_{0}| > \pi/2 $, it is natural to ask whether it is possible to induce the group to reach it by means of an external action. In this sense we introduce the notion of \emph{organization via intervention}.  We consider the system \eqref{system} of $N$ interacting agents, in which the dynamics of every agent is additionally subject to the action of an external field. Admissible controls, accounting for the external field, are measurable functions $u=(u_{1},\ldots,u_{N}):[0,+\infty) \to \RN$ satisfying the $\ell_1^N-\ell_2^d$-norm constraint
\begin{equation}\label{const_cont}
\sum_{i=1}^{N}\|u_{i}(t)\| \leq M,
\end{equation}
for every $t >0$, for a given positive constant $M$. The time evolution of the state is governed by 
\begin{equation}\label{control}
\left\{
\begin{split}
\dot{x}_i(t)&=v_i(t) , \\
\dot{v}_i(t)&=\frac{1}{N}\sum_{j=1}^N a(\Vert x_j(t)-x_i(t)\Vert) (v_j(t)-v_i(t))+ u_{i}(t) ,
\end{split}
\right.
\end{equation}
for $i=1,\ldots,N$, and $x_i\in\R^d$, $v_i\in\R^d$.
In matrix notation, the above system can be written as
\begin{equation}
\left\{ \begin{split}
\dot x &= v \\
\dot v &= - L_{x} v + u,
\end{split}\right.
\end{equation} 
where $L_{x}$ is the Laplacian defined in Section \ref{sec:CSmodel}.
\\

Our aim is then to find admissible controls steering the system to the consensus region in finite time.  
In particular, we shall address the question of quantifying the {\it minimal amount of intervention} one external policy maker should use on the system in order to lead it to consensus,
and we formulate a practical strategy to approach optimal interventions.
Let us mention another conceptually similar approach
to our consensus control, i.e., the mean-field game, introduced by Lasry and Lions \cite{lali07}, and independently in the optimal control community under the name
Nash Certainty Equivalence (NCE) within the work \cite{HCM03}, later greatly popularized within consensus problems, 
for instance in \cite{NCM10,NCM11}. The first fundamental difference with our work is that in (mean-field) games, each individual agent is competing {\it freely} with
the others towards the optimization of its individual goal, as for instance in the financial market, and the emphasis is in the characterization of Nash equilibria.
Whereas in our model we are concerned with the optimization of the intervention of an external policy maker or coordinator endowed with rather limited resources to help the system to form a pattern, when self-organization
does not realize it autonomously, as it is a case, e.g., in modeling economical policies and government strategies.
Let us stress that in our model we are particularly interested to sparsify the control towards most effective results, and also that such an economical concept does not appear anywhere in the literature
when we deal with large particle systems.
\\

Our first approach towards sparse control will be a {\it greedy} one, in the sense that we will design a {\it feedback control} which will optimize instantaneously three fundamental quantities:
\begin{itemize}
\item[(i)] it has the minimal amount of components active at each time;
\item[(ii)] it has the minimal amount of switchings equispaced in time within the finite time interval to reach the consensus region; 
\item[(iii)] it maximizes at each switching time the rate of decay of the Lyapunov functional measuring the distance to the consensus region.
\end{itemize}
This approach models the situation where the external policy maker is actually not allowed to predict future developments and has to make optimal decisions based on instantaneous configurations.
Note that a componentwise sparse feedback control as in (i) is  more realistic and convenient in practice   than a control simultaneously active on more or even all agents,  because it requires acting only on at most one agent, at every instant of time.
The adaptive and instantaneous rule of choice of the controls is based on a variational criterion involving $\ell_1^N-\ell_2^d$-norm penalization terms.
Since however such {\it componentwise sparse controls} are likely to be {\it chattering} (see, for instance, Example~\ref{ex:symm}), i.e., requiring an infinite number of changes of the active control component over a bounded interval of time, we will also have to pay attention in deriving control strategies with property (ii), which are as well sparse in time, and we therefore call them \textit{time sparse controls}.
\\
Our second approach is based on a finite time optimal control problem where the minimization criterion is a combination of the distance from consensus and of the $\ell_1^N-\ell_2^d$-norm of the control. Such an optimization models the situation where the policy maker is actually allowed to make deterministic future predictions of the development. We show that componentwise sparse solutions are again likely to be the most favorable. 
\\

The rest of the paper is organized as follows: Section \ref{sec:stable} is devoted to establishing sparse feedback controls stabilizing System \eqref{control} to consensus. We investigate the construction of componentwise and time sparse controls.
In Section \ref{sec:optimality} we discuss in which sense the proposed sparse feedback controls have actually optimality properties and we address a general notion of complexity for consensus problems. 
In Section \ref{sec:controll} we  we combine the results of the previous sections with a local controllability result near the consensus manifold in order to prove global sparse controllability of Cucker--Smale consensus models. 
We study the sparsity features of solutions of a finite time optimal control of Cucker--Smale consensus models in Section \ref{sec:optimal} and we establish that the lacunarity of their sparsity is related to the codimension of certain manifolds.
The paper is concluded by an Appendix which collects some of the merely technical results of the paper.

\section{Sparse Feedback Control of the Cucker--Smale Model}\label{sec:stable}

\subsection{A first result of stabilization}
Note first that if the integral $\int_{0}^{\infty} a(r) dr$ diverges then every pair $(X, V)>0$ satisfies~\eqref{eq:CScondition}, in other words the interaction between the agents is so strong that the system will reach the consensus no matter what the initial conditions are. In this section we are interested in the case where  consensus does not arise autonomously therefore throughout this section we will assume that
$$
a\in L^{1}(0,+\infty).
$$
As already clarified in Lemma \ref{lem:app1} the quantity $V(t)$ is actually a Lyapunov functional for the uncontrolled System \eqref{system}. For the controlled System \eqref{control} such quantity actually becomes dependent on the choice
of the control, which can nevertheless be properly optimized. As a first relevant and instructive  observation we prove that an appropriate choice of the control law can always stabilize the system to consensus. 

\begin{proposition}\label{prop:stable}
For every $M>0$ and initial condition $(x_{0},v_{0}) \in \RN \times \RN$, the feedback control defined pointwise in time by $u(t) = - \alpha v_{\perp}(t)$, with $0< \alpha \leq \frac{M}{N\sqrt{B(v_{0},v_{0})}}$, satisfies the constraint \eqref{const_cont} for every $t\geq 0$ and stabilizes the system \eqref{control} to consensus in infinite time.
\end{proposition}
\begin{proof}
Consider the solution of \eqref{control} with initial data $(x_{0}, v_{0})$ associated with the feedback control $u = - \alpha v_{\perp}$, with $0 < \alpha \leq \frac{M}{N\sqrt{B(v_{0},v_{0})}}$.
Then, by non-negativity of $L_x$, 
\begin{align*}
\frac{d}{dt} V(t) &= \frac{d}{dt} B(v(t),v(t))\\
&= -2 B(L_{x}v(t), v(t)) + 2 B(u(t),v(t))\\
&\leq 2 B(u(t),v(t)) \\
& = -2 \alpha B(v_{\perp}(t), v_{\perp}(t) )\\
& = -2 \alpha V(t).
\end{align*}
Therefore  $V(t) \leq e^{- 2\alpha t}V(0)$ and $V(t)$ tends to $0$  exponentially fast as $t \to \infty$. 
Moreover
\begin{align*}
\sum_{i=1}^{N} \|u_{i}(t)\| & \leq \sqrt{N} \sqrt{\sum_{i=1}^N \|u_{i}(t)\|^{2}}
= \alpha \sqrt{N}  \sqrt{\sum_{i=1}^N \|v_{\perp_{i}}(t)\|^{2}}
= \alpha N\sqrt{V(t)}
\leq \alpha N \sqrt{V(0)} = M,
\end{align*}
and thus the control is admissible.
\end{proof}

In other words the system \eqref{const_cont}-\eqref{control} is semi-globally feedback stabilizable.
Nevertheless this result has  a merely theoretical value: the feedback stabilizer $u = - \alpha v_{\perp}$ is not convenient for practical purposes since it requires to act at every instant of time on all the agents in order to steer the system to consensus, which may  require a large amount of instantaneous communications. In what follows we address the design of more economical and practical feedback controls which can be both componentwise and time sparse.

\subsection{Componentwise sparse feedback stabilization}
We introduce here a variational principle leading to a componentwise sparse stabilizing feedback law. 

\begin{definition}\label{def:variational}
For every $M>0$ and every $(x,v) \in \RN\times\RN$, let $U(x,v)$ be defined as the set of solutions of the variational problem
\begin{equation}\label{eq:variational}
\min \left( B (v, u) + \gamma(B(x,x)) \frac{1}{N} \sum_{i=1}^{N}\|u_{i}\| \right)  \quad \mbox{ subject to } \sum_{i=1}^{N}\|u_{i}\| \leq M\,,
\end{equation}
where 
\begin{equation}\label{defgamma}
\gamma(X) = \int_{\sqrt{X}}^{\infty} a (\sqrt{2N}r) dr.
\end{equation}
\end{definition}

The meaning of~\eqref{eq:variational} is the following. 
Minimizing the component $B(v,u)=B( v_{\perp}, u)$ means that, at every instant of time, the control $u \in U(x,v)$ is of the form $u = - \alpha \cdot v_\perp$, for some $\alpha = (\alpha_1, \dots, \alpha_N)$ sequence of nonnegative scalars.
Hence it acts as an additional force which pulls the particles towards having the same mean consensus parameter, as highlighted by the proof of Proposition \ref{prop:stable}. Imposing additional $\ell_1^N-\ell_2^d$-norm constraints has the function of enforcing {\it sparsity}, i.e., most of the $\alpha_i's$ will turn out to be zero, as we will in more detail clarify below. Eventually, the threshold $\gamma(X)$ is chosen in such a way that  when the control switches-off the
criterion \eqref{eq:CScondition} is fulfilled.
Let us stress that the choice of the $\ell_1^N$-norm minimization has the relevant advantage with respect
to other  potentially sparsifying constraints, such that, e.g., $\sqrt{\sum_{i=1}^N \|u_i\|^2}$, to reduce the variational principle \eqref{eq:variational} to a very simple separable scalar optimization.

The componentwise sparsity feature of feedback controls $u(x,v) \in U(x,v)$ is analyzed in the next remark, where we make explicit the set $U(x,v)$ according to the value of $(x,v)$ in a partition of the space $\RN\times\RN$.

\begin{remark}\label{rk:Uxv}
First of all, it is easy to see that, for every $(x,v)\in\RN\times\RN$ and every element $u(x,v)\in U(x,v)$ there exist nonnegative real numbers $\alpha_i$'s such that
\begin{equation}\label{eq:distr}
u_{i}(x,v) = 
\left\{
\begin{split}
&0 \quad &\mbox{ if } v_{\perp_{i}} =0,\\
 &- \alpha_{i} \frac{v_{\perp_{i}}}{\|v_{\perp_{i}}\|} &  \quad \mbox{ if } v_{\perp_{i}} \neq 0,
\end{split}
\right.
\end{equation}
where $0\leq \sum_{i=1}^{N}\alpha_{i} \leq M$.
\\
The componentwise sparsity of $u$ depends on the possible values that the $\alpha_i$'s may take in function of $(x,v)$. Actually, the space $\RN\times\RN$ can be partitioned in the union of the four disjoint subsets $\mathcal{C}_{1}, \mathcal{C}_{2}, \mathcal{C}_{3}$, and $\mathcal{C}_{4}$ defined  by
\begin{itemize}
\item[$\mathcal{C}_{1} =$] $\{(x,v)\in \RN\times\RN \ \vert\ {\max_{1\leq i \leq N} \|v_{\perp_{i}}\| < \gamma(B(x,x)) } \}$,
\item[$\mathcal{C}_{2} =$] $\{(x,v)\in \RN\times\RN \ \vert\ {\max_{1\leq i \leq N} \|v_{\perp_{i}}\| = \gamma(B(x,x)) }\}$,
\item[$\mathcal{C}_{3} =$] $\{(x,v)\in \RN\times\RN \ \vert\  {\max_{1\leq i \leq N} \|v_{\perp_{i}}\| > \gamma(B(x,x)) }$ and there exists a unique $i \in \{1,\ldots,N\}$ such that $\|v_{\perp_{i}}\| > \|v_{\perp_{j}}\|$ for every $j \neq i\}$,
\item[$\mathcal{C}_{4} =$] $\{(x,v)\in \RN\times\RN \ \vert\ {\max_{1\leq i \leq N} \|v_{\perp_{i}}\| > \gamma(B(x,x)) }$ and there exist $k \geq 2$ and $i_{1},\ldots, i_{k} \in \{1,\ldots,N\}$ such that 
$\|v_{\perp_{i_{1}}}\| = \cdots = \|v_{\perp_{i_{k}}}\|$ and
$\|v_{\perp_{i_{1}}}\| > \|v_{\perp_{j}}\|$ for every $j \notin \{i_{1},\ldots, i_{k}\}\}$.
\end{itemize}
The subsets  $\mathcal{C}_{1}$ and $\mathcal{C}_{3}$ are open, and the complement $(\mathcal{C}_{1} \cup \mathcal{C}_{3})^c$ has Lebesgue measure zero. Moreover for every $(x,v) \in \mathcal{C}_{1}\cup \mathcal{C}_{3} $, the set $U(x,v)$ is single valued and its value is a sparse vector with at most one nonzero component. More precisely, one has $U|_{\mathcal{C}_{1}} = \{0\}$ and 
$U|_{\mathcal{C}_{3}} = \{(0,\ldots,0, -Mv_{\perp_{i}}/\|v_{\perp_{i}}\|,0,\ldots,0)\}$ for some unique $i\in\{1,\ldots,N\}$.
\\
If $(x,v) \in \mathcal{C}_{2}\cup \mathcal{C}_{4}$ then a control in $U(x,v)$ may not be sparse: indeed in these cases the set $U(x,v)$ consists of all $u=(u_{1},\ldots,u_{N}) \in \RN$ such that $u_{i}=-\alpha_{i}v_{\perp_{i}}/\|v_{\perp_{i}}\|$ for every $i=1,\dots,N$, where the $\alpha_i$'s are nonnegative real numbers such that $0\leq \sum_{i=1}^{N}\alpha_{i} \leq M$ whenever $(x,v) \in \mathcal{C}_{2}$, and $\sum_{i=1}^{N}\alpha_{i} = M$ whenever $(x,v)\in \mathcal{C}_{4}$.
\end{remark}

By showing that the choice of feedback controls as in Definition \ref{def:variational}  optimizes the Lyapunov functional $V(t)$, we can again prove convergence to consensus.

\begin{theorem}\label{thm:main}
For every $(x,v)\in\RN\times\RN$, and $M>0$, set $F(x,v) = \{(v, -L_{x}v + u)\ \vert\  u\in U(x,v)\}$, where $U(x,v)$ is as in Definition \ref{def:variational}.
Then for every initial pair $(x_{0},v_{0}) \in \RN \times \RN$, the differential inclusion 
\begin{equation}\label{eq:inclusion}
(\dot x,\dot v) \in F(x,v)
\end{equation}
with initial condition $(x(0),v(0)) = (x_{0},v_{0})$ is well-posed and its solutions converge to consensus as $t$ tends to $+\infty$.
\end{theorem}
\begin{remark}\label{remm4}
By definition of the feedback controls $u(x,v)\in U(x,v)$, and from Remark \ref{rk:Uxv}, it follows that, along a closed-loop trajectory, as soon as $V(t)$ is small enough with respect to $\gamma(B(x,x))$ the trajectory has entered the consensus region defined by \eqref{eq:CScondition}. From this point in time no action is further needed to stabilize the system, since Proposition~\ref{prop:condition-for-consensus} ensures then that the system is naturally stable to consensus. 
 Notice that ${\mathcal{C}_{1}}$ is strictly contained in the consensus region and, moreover, every trajectory of the uncontrolled system~\eqref{system} originating in $\mathcal{C}_{1}$ remains in $\mathcal{C}_{1}$ (see Lemma~\ref{lem:invariance} in the Appendix).
Hence when the system enters the region $\mathcal{C}_{1}$,  in which there is no longer need to control,  the control switches-off automatically end it is set to $0$ forever. 
It follows from the proof of Theorem \ref{thm:main} below
that the time $T$ needed to steer the system to the consensus region is not larger than
$\frac{{N}}{M}\left(\sqrt{V(0)} -\gamma(\bar X)\right)$, where $\gamma$ is defined by \eqref{defgamma}, and $\bar X= 2 X(0)  + \frac{{N^{4}}}{2 M^{2}} V(0)^2$.
\end{remark}

\begin{proof}[Proof of Theorem \ref{thm:main}]
 First of all we prove that \eqref{eq:inclusion} is well-posed, by using general existence results of the theory of differential inclusions (see e.g. \cite[Theorem~2.1.3]{AubinCellina}). 
For that we address the following steps:
\begin{itemize}
\item being the  set $F(x,v)$ non-empty, closed, and convex for every $(x,v) \in \RN\times \RN$ (see Remark~\ref{rk:Uxv}), we show that $F(x,v)$ is upper semi-continuous; this will imply local existence of solutions of  \eqref{eq:inclusion};
\item we will then argue the global extension of these solutions for every $t \geq 0$ by the classical theory of ODE's, as it is  sufficient to remark that there exist positive constants $c_{1},c_{2}$ such that $\|F(x,v)\| \leq c_{1} \|v\| + c_{2}$.
\end{itemize}
Let us address the upper semi-continuity of $F(x,v)$, that is 
for every $(x_{0},v_{0})$ and for every $\varepsilon>0$ there exists $\delta >0$ such that
$$
F(B_{\delta}(x_{0},v_{0})) \subset B_{\varepsilon} (F(x_{0},v_{0})),
$$
where $B_{\delta}(y), B_{\varepsilon}(y) $ are the balls of $ \RN\times \RN$ centered in $y$ with radius $\delta$ and $\varepsilon$ respectively.
As the composition of upper semi-continuous functions is upper semi-continuous (see~\cite[Proposition~1.1.1]{AubinCellina}), then it is sufficient to prove that $U(x,v)$ is upper semi continuous. For every $(x,v) \in \mathcal{C}_{1}\cup \mathcal{C}_{3}$, $U(x,v)$ is single valued and continuous. If $(x,v)\in \mathcal{C}_{2}$ then
there exist $i_{1},\ldots,i_{k}$ such that $ \|v_{\perp_{i_{1}}}\| = \cdots = \|v_{\perp_{i_{k}}}\|$ and $\|v_{\perp_{i_{1}}}\| > \|v_{\perp_{l}}\|$ for every $l \notin  \{i_{1},\ldots,i_{k}\}$. If $\delta < \min_{l \notin  \{i_{1},\ldots,i_{k}\}}\left( \|v_{\perp_{i_{1}}}\| - \|v_{\perp_{l}}\| \right)$  then
$U(B_{\delta}(x,v)) =  U(x,v)$ hence, in particular, it is upper semi continuous. With a similar argument it is possible to prove that $U(x,v)$ is upper semi continuous for every $(x,v)\in \mathcal{C}_{4}$.
This establishes the well-posedness of \eqref{eq:inclusion}.

Now, let $(x(\cdot),v(\cdot))$ be a solution of~\eqref{eq:inclusion}. 
Let $T$ the minimal time needed to reach the consensus, that is $T$ is the smallest number such that $V(T) = \gamma(X(T))^2$, with the convention that $T=+\infty$ if the system does not reach consensus. 
For almost every $t \in (0,T)$ then we have $V(t)> \gamma(X(t))^2$.
Thus the trajectory $(x(\cdot),v(\cdot))$ is in $\mathcal{C}_{3}$ or $\mathcal{C}_{4}$ and
there exist indices $i_{1},\ldots, i_{k}$ in $\{1,\ldots,N\}$ such that 
$\|v_{\perp_{i_{1}}}(t)\| = \cdots = \|v_{\perp_{i_{k}}}(t)\|$ and
$\|v_{\perp_{i_{1}}}(t)\| > \|v_{\perp_{j}}(t)\|$ for every $j \notin \{i_{1},\ldots, i_{k}\}$. 
Hence if $u(t) \in U(x(t),v(t))$ then
$$
u_{j}(t) = 
\begin{cases}
-\alpha_{j} \displaystyle\frac{v_{\perp_{j}}(t) }{\|v_{\perp_{j}}(t)\| } \quad  &\mbox{ if } j\in\{i_{1},\ldots,i_{k}\}, \\
0 & \mbox{ otherwise,}
\end{cases}
$$
where $\sum_{j=1}^{k} \alpha_{i_{j}}=M$.
Then,
\begin{align}
\frac{d}{dt} V(t) &= \frac{d}{dt} B(v(t),v(t))\nonumber\\
& \leq 2 B(u(t),v(t))\nonumber\\
& = \frac{2}{{N}} \sum_{i=1}^{N}  \langle u_{i}(t),v_{\perp_{i}}(t) \rangle\nonumber\\
& = -  \frac{2}{{N}}  \sum_{j=1}^{k} \alpha_{i_{j}} \|v_{\perp_{i_{j}}}(t)\| \nonumber\\
& = -2  \frac{M}{{N}}  \|v_{\perp_{i_{1}}}(t)\| \nonumber\\
&\leq -2 \frac{M}{{N}} \sqrt{V(t)}.
\end{align}
For clarity, notice that in the last inequality we used the maximality of $\|v_{\perp_{i_{1}}}(t)\|$ for which 
$$
\frac{N}{N^2}\|v_{\perp_{i_{1}}}(t)\|^2 \geq \frac{1}{N^2} \sum_{j=1}^N \|v_{\perp_{j}}(t)\|^2, 
$$
or
$$
\frac{\sqrt N}{N}\|v_{\perp_{i_{1}}}(t)\| \geq \frac{1}{\sqrt N} \left ( \frac{1}{N} \sum_{j=1}^N \|v_{\perp_{j}}(t)\|^2 \right )^{1/2},
$$
and eventually
$$
- \frac{1}{N}\|v_{\perp_{i_{1}}}(t)\| \leq -\frac{1}{N} \sqrt{V(t)}.
$$
Let $V_{0} = V(0)$ and $X_{0}=X(0)$. It follows from Lemma~\ref{lem:integrate} in Appendix, or simply by direct integration, that
\begin{equation}\label{eq:kappa}
V(t) \leq \left( \sqrt{V_{0}} - \frac{M}{{N}}  t \right)^{2},
\end{equation}
and 
$$
X(t) \leq 2 X_{0}  + \frac{{N^{4}}}{2 M^{2}} V_{0}^{2} = \bar X.
$$ 
Note that the time needed to steer the system in the consensus region is not larger than
\begin{equation} \label{eq:bound}
T_0= \frac{{N}}{M}\left(\sqrt{V_{0}} -\gamma(\bar X)\right),
\end{equation}
and in particular it is finite. Indeed, for almost every $t > T_{0}$ we have
$$
\sqrt{V(t)} < \sqrt{V(T_{0})} \leq  \sqrt{V_{0}} - \frac{M}{N} T_{0} = 
\gamma(\bar X) \leq \gamma(X(t)),
$$
and 
Proposition \ref{prop:condition-for-consensus}, in particular~\eqref{eq:CScondition}, implies that 
the system is in the consensus region.
Finally, for $t$ large enough $\max_{1\leq i \leq N} \|v_{\perp_{i}}\| < \gamma(X(t))$,
then by Lemma~\ref{lem:app1} we infer that $V(t)$ tends to $0$. 
\end{proof}

Within the set $U(x,v)$ as in Definition \ref{def:variational}, which in general does not contain only sparse solutions, 
there are actually selections with maximal sparsity.

\begin{definition}\label{def:u}
We select the \emph{componentwise sparse feedback control} $u^\circ = u^\circ(x,v)\in U(x,v)$ according to the following criterion:
 \begin{itemize}
\item if ${\max_{1\leq i\leq N} \|v_{\perp_{i}}\| \leq  \gamma(B(x,x))^{2}}$, then $u^\circ=0$,
\item if 
$\max_{1\leq i\leq N} \|v_{\perp_{i}}\| >  \gamma(B(x,x))^{2}$ let $j \in \{1,\ldots,N\}$ be the smallest index such that 
$$
\|v_{\perp_{j}}\| =  \max_{1\leq i\leq N} \|v_{\perp_{i}}\|
$$
then 
$$
 u^\circ_{j} = - M \frac{v_{\perp_{j}}}{\|v_{\perp_{j}}\|},\quad\mbox{ and }\quad u^\circ_{i} = 0 \quad \mbox{for every } i\neq j.
$$
\end{itemize}
\end{definition}

The control $u^{\circ}$ can be, in general, highly irregular in time. If we consider for instance a system in which there are two agents with maximal disagreement then the control $u^{\circ}$ switches at every instant from one agent to the other and it is everywhere discontinuous. The natural definition of solution associated with the feedback control $u^{\circ}$ is therefore the notion of sampling solution 
as introduced in~\cite{CLSS}. 

\begin{definition}[Sampling solution]\label{def:samplingsolution}
Let $U \subset \R^{m}$,  $f:\R^{n}\times U \to \R^{n}$ be continuous and locally Lipschitz in $x$ uniformly on compact subset of $\R^{n}\times U$. 
Given a feedback $u:\R^{n}\to U$,  $\tau>0$, and $x_{0} \in \R^{n}$ 
we define the \emph{sampling solution}  
of the differential system 
$$
\dot x = f(x,u(x)),\quad x(0)=x_{0},
$$
as the continuous (actually piecewise $C^1$) function $x:[0,T] \to \mathbb R^n$ solving recursively for $k \geq 0$
$$ 
\dot x (t) = f(x(t), u(x(k\tau))), \quad t\in [k\tau,(k+1)\tau]
$$ 
using as initial value $x(k\tau)$, the endpoint of the solution on the preceding interval, and starting with $x(0)=x_{0}$. We call $\tau$ the \emph{sampling time}. 
\end{definition}

This notion of solution is of particular interest for applications in which a minimal interval of time between two switchings of the control law is demanded.
As the sampling time becomes smaller and smaller the sampling solution of~\eqref{control} associated with our componentwise sparse control $u^\circ$ as defined in Definition~\ref{def:u} 
approaches uniformly a Filippov solution of~\eqref{eq:inclusion}, i.e. an absolutely continuous function satisfying~\eqref{eq:inclusion} for almost every $t$.
In particular we shall prove in Section~\ref{sec:proofthm2} the following  statement.

\begin{theorem}\label{thm:main2}
Let $u^\circ$ be the componentwise sparse control defined in Definition~\ref{def:u}. For every $M>0, \tau>0$, and $(x_{0},v_{0})\in \RN\times\RN$ let $(x_{\tau}(t),v_{\tau}(t))$ be the sampling solution of~\eqref{control} associated with $u^{\circ}$. 
Every closure point of the sequence of trajectories $\left((x_{\tau}(t),v_{\tau}(t))\right)_{\tau>0}$
is a Filippov solution of~\eqref{eq:inclusion}.
\end{theorem}
Let us stress that, as a byproduct of our analysis, we shall eventually construct practical feedback controls which are both {\it componentwise} and {\it time sparse}.

\subsection{Time sparse feedback stabilization}

Theorem~\ref{thm:main} gives the existence of a feedback control whose behavior may be, in principle, very complicated and that may be nonsparse.  In this section we are going to exploit the variational principle~\eqref{eq:variational} to give an explicit construction of a piecewise constant and componentwise sparse control steering the system to consensus. The idea is to take a selection of a feedback in $U(x,v)$ which has {\it at most one nonzero component} for every $(x,v) \in \RN\times\RN$, as in Definition \ref{def:u}, and then sample it to avoid \emph{chattering} phenomena (see, e.g., \cite{borisov}).

\begin{theorem}\label{prop:piecewise}
Fix $M>0$ and consider the control $u^\circ$ law given by Definition~\ref{def:u}. Then for every initial condition $(x_{0},v_{0}) \in \RN \times \RN$  there exists $\tau_0> 0$ small enough,
such that for all $\tau \in (0, \tau_0]$ the sampling solution of~\eqref{control} associated with the control $u^\circ$, the sampling time $\tau$, and initial pair $(x_{0},v_{0})$
reaches the consensus region in finite time.
\end{theorem}

\begin{remark}\label{rem:maxtime}
The maximal sampling time $\tau_0$ depends on the number of agents $N$, the $\ell_1^N-\ell_2^d$-norm bound $M$ on the control,  the initial conditions $(x_0,v_0)$, and the
rate of communication function $a(\cdot)$. The precise  bounding condition \eqref{eq:tau1} 
is given in the proof below.  Moreover, as in Remark \ref{remm4}, the sampled control is switched-off as soon as the sampled trajectory enters the 
region $\mathcal{C}_{1}$. In particular the systems reaches the consensus region defined by \eqref{eq:CScondition}
within time $T \leq T_0= \frac{2 N}{M}(\sqrt{V(0)} - \gamma(\bar X))$, where $\bar X = 2 B(x_{0},x_{0}) + \frac{2N^{4}}{M^{2}} B(v_{0},v_{0})^{2}$. The control is then set to zero in a time that is not larger than 
$ \frac{2 \sqrt{N}}{M}( \sqrt{N}\sqrt{V(0)} - \gamma(\bar X))$.
\end{remark}

{
\begin{proof}[Proof of Theorem \ref{prop:piecewise}]
 Let
$$
\bar X = 2 B(x_{0},x_{0}) + \frac{2N^{4}}{M^{2}} B(v_{0},v_{0})^{2}.
$$
and let $\tau>0 $ satisfy the following condition
\begin{equation}\label{eq:tau1}
 \tau \left(a(0) (1+\sqrt{N}) \sqrt{B(v_{0},v_{0})} + M\right) + \tau^{2} 2 a(0)M    \leq \frac{\gamma(\bar X)}{2}.
\end{equation}
Denote by $(x,v)$ the sampling solution of System~\eqref{control} associated with the control $u^\circ$, the sampling time $\tau$, and the initial datum $(x_{0},v_{0})$.
 Here $[\cdot]$ denotes the integer part of a real number.
Let $\tilde u(t) = u^\circ(x(\tau[t/\tau]), v(\tau [t/\tau ]))$ and denote for simplicity $u^\circ(t) = u^\circ(x(t),v(t))$, then
$\tilde u(t) = u^\circ(\tau[t/\tau ])$.

Let $T >0$  be the smallest time such that $\sqrt{V(T)} = \gamma(\bar X)$ 
with the convention that  $T = +\infty$ if  $\sqrt{V(t)} > \gamma(\bar X)$ for every $t\geq0$.
(If $\sqrt{V(0)} = \gamma(\bar X) \leq \gamma(X(0))$ the system is in the consensus region and there is nothing to prove.)
For almost every $t \in [0, T]$, and by denoting $n = [ {t}/{\tau}]$, we have
\begin{align}
\frac{d}{dt} V(t) &= \frac{d}{dt} B (v(t),v(t))\nonumber\\
& \leq 2 B(\tilde u(t), v(t)) \nonumber\\
& = 2 B(u^\circ(n\tau), v(t)). \label{eq:1101} 
\end{align}
Let $i$ in $\{1,\ldots,N\}$ be the smallest index  such that
$
\|v_{\perp_{i}}(n\tau)\| \geq \|v_{\perp_{k}}(n\tau)\|
$
for every $k \neq i$, 
so that  $u^\circ_{i}(n\tau) = -M v_{\perp_{i}}(n\tau) / \|v_{\perp_{i}}(n\tau)\|$ and $u^\circ_{k}(n\tau) =0$ for every $k \neq i$.
Then~\eqref{eq:1101} reads
\begin{equation}\label{firstestim}
\frac{d}{dt} V(t) \leq - \frac{2M}{N} \phi(t)\,,
\end{equation}
where 
$$
\phi(t) = \frac{\langle v_{\perp_{i}}(n\tau),  v_{\perp_{i}}(t)\rangle}{\| v_{\perp_{i}}(n\tau)\|}.
$$ 
Note that
\begin{equation}\label{estbelow}
\phi(n\tau) = \|v_{\perp_{i}}(n\tau)\| \geq \sqrt{V(n\tau)}. 
\end{equation} 
Moreover, by observing $\|v_{\perp_{i}}(t)\|^2 \leq N \left ( \frac{1}{N} \sum_{j=1}^N  \|v_{\perp_{j}}(t)\|^2 \right )$,
we have also the following estimates from above
\begin{equation}\label{estbelow2}
- \phi(t) \leq  \|v_{\perp_{i}}(t)\| \leq  \sqrt{N} \sqrt{V(t)}.
\end{equation}
We combine \eqref{estbelow2} with \eqref{firstestim} to obtain
$$
\frac{d}{dt} V(t) \leq \frac{2M}{\sqrt N} \sqrt{V(t)},
$$
and, by integrating between $s$ and $t$, we get 
\begin{equation}\label{eq:9292}
\sqrt{V(t)} \leq \sqrt{V(s)} + (t-s)\frac{M}{\sqrt{N}}.
\end{equation}
Now, we prove that $V$ is decreasing in $[0,T]$.
Notice that
\begin{align*}
\frac{d}{dt} v_{\perp_{i}}(t) 
&= \frac{1}{N} \sum_{k\neq i} a(\|x_{k} - x_{i}\|)(v_{\perp_{k}}(t) - v_{\perp_{i}}(t))+ \tilde u_{i}  - \frac{1}{N}\sum_{\ell=1}^{N} \tilde u_{\ell} \\
&= \frac{1}{N} \sum_{k\neq i} a(\|x_{k} - x_{i}\|)(v_{\perp_{k}}(t) - v_{\perp_{i}}(t)) - M \frac{N-1}{N} \frac{v_{\perp_{i}}(n\tau)}{\|v_{\perp_{i}}(n\tau)\|}.
\end{align*}
Moreover, observing that by Cauchy-Schwarz 
$$\sum_{k=1}^{N} \|v_{\perp_{k}}\| \leq \sqrt{N} \left (\sum_{k=1}^{N} \|v_{\perp_{k}}\|^2 \right)^{1/2} = N \left ( \frac{1}{N}\sum_{k=1}^{N} \|v_{\perp_{k}}\|^2 \right)^{1/2},
$$
we have the following sequence of estimates
$$
\frac{1}{N} \sum_{k\neq i}\|v_{\perp_{k}}(t) - v_{\perp_{i}}(t)\| \leq 
\frac{1}{N} \sum_{k\neq i}\|v_{\perp_{k}}(t)\| + \| v_{\perp_{i}}(t)\| 
 \leq\frac{1}{N} \sum_{k=1}^{N} \|v_{\perp_{k}}(t)\| + \sqrt{N} \sqrt{V(t)}
 \leq (1+\sqrt{N})\sqrt{V(t)}.
$$
Hence
\begin{align*}
\frac{d}{dt} \phi(t) &= 
 \frac{\langle v_{\perp_{i}}(n\tau),  \dot v_{\perp_{i}}(t)\rangle}{\| v_{\perp_{i}}(n\tau)\|}\\
 &= \frac{1}{N\|v_{\perp_{i}}(n\tau)\|} \sum_{k\neq i} a(\|x_{k} - x_{i}\|)
 \langle v_{\perp_{k}}(t) - v_{\perp_{i}}(t), v_{\perp_{i}}(n\tau) \rangle - \frac{N-1}{N}M\\
 &\geq - \frac{1}{N} a(0)\sum_{k\neq i}\|v_{\perp_{i}}(t) - v_{\perp_{k}}(t)\|  - M\\
 &\geq -a(0)(1+\sqrt{N}) \sqrt{V(t)}  - M.
\end{align*}
By mean-value theorem  there exists $\xi \in [n\tau,t]$ such that
$$
\phi(t) \geq \phi(n\tau) - (t-n\tau) \left(a(0)(1+\sqrt{N}) \sqrt{V(\xi)}  + M\right).
$$
Then, using the growth estimate~\eqref{eq:9292} on $\sqrt{V}$, and estimating 
$\sqrt{V(\xi)}$  from above by $\sqrt{V(n\tau)} + \tau {M}/{\sqrt{N}}$, we have
$$
\phi(t) \geq \phi(n\tau) -  \tau \left(a(0)(1+\sqrt{N}) \sqrt{V(n\tau)} +M \right) - \tau^{2} 2 a(0) M .
$$
Plugging this latter expression again in~\eqref{firstestim} and using \eqref{estbelow}, we have 
\begin{equation}\label{eq:2210}
\frac{d}{dt} V(t)  \leq - \frac{2M}{N} \left(\sqrt{V(n\tau)}
-  \tau \left(a(0)(1+\sqrt{N}) \sqrt{V(n\tau)} + M \right) - \tau^{2} 2 a(0) M \right).
\end{equation}

We prove by induction on $n$ that $V(t)$ is decreasing on $[0, T]$. 
Let us start on $[0,\tau]$ by assuming $\sqrt{V(0)} > \gamma(\bar X)$, otherwise we are already in the consensus region and
there is nothing further to prove.
By~\eqref{eq:2210} and using the condition~\eqref{eq:tau1} on $\tau$, we infer 
\begin{align}
\frac{d}{dt} V(t) &\leq  -\frac{2M}{N} \left( \sqrt{V(0)}  -  \tau \left(a(0)(1+\sqrt{N}) \sqrt{V(0)} + M\right) - \tau^{2} 2 a(0) M \right)
\nonumber\\ 
&\leq -\frac{2M}{N}\left( \gamma(\bar X)  -  \frac{\gamma(\bar X)}{2} \right)\nonumber\\ 
& = -\frac{M}{N} \gamma(\bar X) < 0.
\label{eq:6473}
\end{align}
Now assume that $V$ is actually decreasing on $[0,n\tau]$, $n\tau < T$, and thus $\sqrt{V(n\tau)} > \gamma(\bar X)$. Let us prove that $V$ is decreasing also on $[n\tau, \min\{T,(n+1)\tau\}]$. 
For every $t\in (n\tau, \min\{T,(n+1)\tau\})$, we can recall again equation \eqref{eq:2210}, and use the inductive hypothesis of monotonicity
for which $\sqrt{V(0)} \geq \sqrt{V(n \tau)}$, and  the condition~\eqref{eq:tau1} on $\tau$ to show
\begin{align}
\frac{d}{dt} V(t)  &\leq - \frac{2M}{N}\left(\sqrt{V(n\tau)} -  \tau \left(a(0)(1+\sqrt{N}) \sqrt{V(n\tau)} + M \right) - \tau^{2} 2 a(0) M \right) \nonumber \\
&\leq -\frac{2M}{N}\left( \gamma(\bar X)  -  \tau \left(a(0)(1+\sqrt{N}) \sqrt{V(0)} + M\right) - \tau^{2} 2 a(0) M \right) \nonumber \\
&\leq -\frac{M}{N} \gamma(\bar X) < 0.\nonumber 
\end{align}
This proves that $V$ is decreasing on $[0,T]$.

Let us now use a bootstrap argument to derive an algebraic rate of convergence towards the consensus region.
For every $t \in (0,T)$ by using~\eqref{eq:2210}, the fact that $V$ is decreasing, and the condition~\eqref{eq:tau1} on $\tau$ we have
\begin{align*}
\frac{d}{dt} V(t) &\leq - \frac{2M}{N} \left(\sqrt{V(n\tau)} -  \frac{\gamma(\bar X)}{2}\right)\\
&\leq - \frac{M}{N} \sqrt{V(t)}.
\end{align*}
Then
$$
\sqrt{V(t)} \leq \sqrt{V(0)} - \frac{M}{2N}t,
$$
for every $t\in [0,T]$. Finally we get $T \leq 2 N(\sqrt{V(0)} - \gamma(\bar X))/M$. Moreover, since 
$\max_{1\leq i\leq N} \|v_{\perp_{i}}\| \leq \sqrt{N}\sqrt{V(t)}$, then the control switches off after a time smaller than or equal to $2 \sqrt{N}(\sqrt{N}\sqrt{V(0)} - \gamma(\bar X))/M$.
\end{proof}
\subsection{Componentwise sparse selections are absolutely continuous solutions}\label{sec:proofthm2}
We are now ready to prove Theorem~\ref{thm:main2}.

\begin{proof}[Proof of Theorem~\ref{thm:main2}]
Denote by $z=(x,v)$ an element of $\RN\times \RN$. Fix $z_{0} = (x_{0},v_{0}) \in \RN\times \RN$. 
Let $\tau_{0}$ be the sampling time in Theorem~\ref{prop:piecewise} determining a sampling solution converging to consensus.
For every $n > 1/\tau_{0}$ consider the sampling solution $z_{n}$ of \eqref{control} associated with the feedback $u^{\circ}$, the sampling time $1/n$, and the initial datum $z_{0}$.
Let $u_{n}(t) = u^\circ(z_{n}([nt]/n))$ and 
let $\mathfrak{u}_{n}(t)$ be the extension of $u_{n}(t)$ to $\RN \times \RN$ which is zero on the first $dN$ components and equal to $u_{n}(t)$ on the last $dN$.
If $f(z)= (v, -L_{x} v)$ we have that
$$
z_{n}(t) = z_{0} + \int_{0}^{t} \left (f(z_{n}(s)) + \mathfrak{u}_{n}(s) \right )ds.
$$
For a suitable constant $\alpha>0$, the linear growth estimate  $\|f(z)\| \leq \alpha (\|z\| + 1)$ holds, so that, in particular we have
$$
\|z_{n}(t)\| \leq \frac{e^{\alpha t} (\alpha\|z_{0}\| + \alpha+M) - \alpha -M}{\alpha},
$$
where the bound is uniform in $n$. 
Let, as in Remark~\ref{rem:maxtime}, $T = \frac{2 \sqrt{N}}{M}( \sqrt{N}\sqrt{B(v_{0},v_{0})} - \gamma(\bar X))$, where 
$\bar X = 2 B(x_{0},x_{0}) + \frac{2N^{4}}{M^{2}} B(v_{0},v_{0})^{2}$. Note that $T$ does not depend on $n$.
Therefore the sequence of continuous functions $(z_{n})_{n \in \mathbb N}$ is  equibounded by the constant 
$$
C=\frac{e^{\alpha T} (\alpha\|z_{0}\| + \alpha+M) - \alpha -M}{\alpha}.
$$
The sequence $(z_{n})_{n \in \mathbb N}$ is also equicontinuous. Indeed 
\begin{equation}\label{eq:0392}
\|z_{n}(t) - z_{n}(s)\|  \leq \int_{s}^{t} \left ( \|f(z_{n}(\xi))\| + M  \right )d\xi \leq (t-s)(\alpha (C + 1) + M)
\end{equation}
for every $n$.
For every $\varepsilon > 0$,  if $\delta = \varepsilon /(\alpha (C + 1) + M)>0$ then for every $n$ one has
$\|z_{n}(t) - z_{n}(s)\| < \varepsilon$ whenever $|t-s|<\delta$.
By Ascoli--Arzel\`a Theorem, up to subsequences, $z_{n}$ converges uniformly to an absolutely continuous function $z$ as $n$ tends to infinity. 

Let us prove that $z$ is a Filippov solution of~\eqref{eq:inclusion}.
By continuity $f(z_{n}(t))$ converges to $f(z(t))$ for almost every $t$. 
Since 
$$
\int_{s}^{t}\mathfrak{u}_{n}(\xi) d\xi = z_{n}(t) - z_{n}(s) - \int_{s}^{t} f(z_{n}(\xi))  d\xi,
$$
then, by~\eqref{eq:0392}, Dunford--Pettis Theorem (see, for instance, \cite[Theorem IV.29]{brezis}) applies and $u_{n}$ converges weakly in $L^{1}$ to an admissible control $u$ as $n$ tends to infinity. Denote, as above, by $\mathfrak{u}$ the extension of $u$ to $\RN \times \RN$ which is zero on the first $dN$ components.
By the dominated convergence Theorem, the limit function $z$ satisfies
$$
z(t) = z_{0} + \int_{0}^{t} \left ( f(z(s)) + \mathfrak{u}(s) \right )ds.
$$
The map $z \to U(z)$ is actually upper hemicontinuous in the sense of \cite[Definition~1 p. 59] {AubinCellina}because $U(z)$ is a polytope which is just continuously perturbed and at most looses dimensionality whenever we continuously
perturb $z$. In particular, it can never gain dimensionality. Moreover, all the conditions of \cite[Theorem~1 p. 60]{AubinCellina} are fulfilled for $(x,y)=(z,u)$ and $F(x) = U(z)$ in its notations, implying that  $u \in U(z)$, i.e., it is a solution of the variational problem~\eqref{eq:variational} and $z$ is therefore a Filippov solution of the differential inclusion~\eqref{eq:inclusion}.
\end{proof}

\section{Sparse is Better}\label{sec:optimality}

\subsection{Instantaneous optimality of componentwise sparse controls}
The componentwise sparse control $u^\circ$ of Definition \ref{def:u} corresponds to the strategy of acting, at each switching time, on the agent whose consensus parameter is farthest from the mean and to steer it to consensus.
Since this control strategy is designed to act on at most one agent at each time, we claim that in some sense it is instantaneously the ``best one''.
To clarify this notion of {\it instantaneous optimality} which also explains its {\it greedy} nature, we shall compare this strategy with all other feedback strategies $u(x,v)\in U(x,v)$ and discuss their efficiency in terms of the instantaneous decay rate of the functional $V$.

\begin{proposition}\label{prop:sparseisbetter}
The feedback control $u^\circ(t)=u^\circ (x(t),v(t))$ of Definition \ref{def:u}, associated with the solution $((x(t),v(t))$ of Theorem \ref{thm:main2}, is a minimizer of 
$$
\mathscr R(t,u)=\frac{d}{dt} V(t),
$$
over all possible feedback controls in $U(x(t),v(t))$. In other words, the feedback control $u^\circ$ is the best choice in terms of the rate of convergence to consensus.
\end{proposition}

\begin{proof}
Consider
\begin{align*}
\frac{d}{dt} V(t) & = \frac{1}{N} \frac{d}{dt}  \sum_{i=1}^{N}  \|v_{\perp_{i}}\|^{2}\\
&=  \frac{2}{N}  \sum_{i=1}^{N} \langle \dot v_{\perp_{i}}, v_{\perp_{i}} \rangle\\
&= \frac{2}{N^{2}} \sum_{i=1}^{N}\sum_{j=1}^{N} a(\|x_{i} - x_{j}\|) (\langle v_{\perp_{i}}, v_{\perp_{j}} \rangle - \|v_{\perp_{i}}\|^{2}) + 
\frac{2}{N} \sum_{i=1}^{N}\langle u^\circ_{i} - \frac{1}{N}\sum_{j=1}^{N}u^\circ_{j} ,v_{\perp_{i}} \rangle.
\end{align*}
Now consider controls $u_{1},\ldots,u_{N}$ of the form~\eqref{eq:distr}, then
\begin{align*}
\sum_{i=1}^{N}\langle u_{i} - \frac{1}{N}\sum_{j=1}^{N}u_{j} ,v_{\perp_{i}} \rangle &=
- \sum_{\{i\ \vert\  v_{\perp_{i}}\neq 0\}} \alpha_{i} \|v_{\perp_{i}}\| + 
\frac{1}{N} \sum_{\{i\ \vert\  v_{\perp_{i}}\neq 0\}}\ \ \sum_{\{j\ \vert\  v_{\perp_{j}}\neq 0\}} 
\alpha_{j}
\frac{\langle v_{\perp_{i}}, v_{\perp_{j}}\rangle}{\|v_{\perp_{j}}\|}\\
&= - \sum_{\{i\ \vert\  v_{\perp_{i}}\neq 0\}} \alpha_{i} \|v_{\perp_{i}}\| +
\frac{1}{N} \sum_{\{j\ \vert\  v_{\perp_{j}}\neq 0\}} \left \langle
\underbrace{ 
\sum_{\{i\ \vert\  v_{\perp_{i}}\neq 0\}} v_{\perp_i}}_{=0}, \alpha_j \frac{ v_{\perp_{j}}}{\| v_{\perp_{j}} \| } \right \rangle\\
&= - \sum_{\{i\ \vert\  v_{\perp_{i}}\neq 0\}} \alpha_{i} \|v_{\perp_{i}}\|
\end{align*}
since, by definition, $\sum_{i=1}^N v_{\perp_{i}} \equiv 0$. 
Then maximizing the decay rate of $V$ is equivalent to solve
\begin{equation}\label{eq:mejo}
\max \sum_{j=1}^{N} \alpha_{j}\|v_{\perp_{j}}\|, \quad \mbox{ subject to } \alpha_j\geq 0,\ \sum_{j=1}^{N} \alpha_{j} \leq M.
\end{equation}
In fact, if the index $i$ is such that $\|v_{\perp_{i}}\| \geq \|v_{\perp_{j}}\|$ for $j \neq i$ as in the definition of $u^\circ$, then
$$\sum_{j=1}^{N} \alpha_{j}\|v_{\perp_{j}}\| \leq \|v_{\perp_{i}}\| \sum_{j=1}^{N} \alpha_{j} \leq M \|v_{\perp_{i}}\|.$$
Hence the control $u^\circ$ is a maximizer of \eqref{eq:mejo}. This variational problem has a unique solution whenever there exists a unique $i\in \{1,\ldots,N\}$ such that $\|v_{\perp_{i}}\| > \|v_{\perp_{j}}\|$ for every $j \neq i$.
\end{proof}

This result is somewhat surprising with respect to the perhaps more intuitive strategy of activating controls on more agents or even (although not realistic) all the agents at the same time as given in Proposition \ref{prop:stable}. 
This can be viewed as a mathematical description of the following general principle: 
\begin{center}
\fbox{\parbox[c][1.7cm]{15.1cm}{\it A policy maker, who is not allowed to have prediction on future developments, should always consider more favorable to intervene with stronger actions on the fewest possible instantaneous optimal leaders than
trying to control more agents with minor strength.}}
\end{center}

\begin{example}\label{ex:symm}
The limit case when the action of the sparse stabilizer and of a control acting on all agents are equivalent is represented by the symmetric case in which there exists a set of indices $\Lambda=\{i_1, i_2, \dots, i_k\}$ such that $\|v_{\perp_{i_\ell}}\| =\|v_{\perp_{i_m}}\|$ and $ \|v_{\perp_{i_\ell}}\| > \|v_{\perp_{j}}\|$ for every $j \notin \Lambda$ and for all $i_\ell,i_m \in \Lambda$. In this case, indeed, the equation \eqref{eq:mejo} of the proof of Proposition \ref{prop:sparseisbetter} has more solutions. Consider four agents on the plane $\R^{2}$ with initial main states
$x_{1}(0) = (-1,0), x_{2}(0) = (0,1), x_{3}(0) = (1,0), x_{4}(0) = (0,-1)$ and consensus parameters
 $v_{1}(0)=(-1,0), v_{2}(0)=(0,1), v_{3}(0)=(1,0), v_{4}(0)=(0,-1)$. Let the interaction function be
 $a(x) = 2/(1+x^{2})$ and the bound on the control be $M=1$. In Figure~\ref{fig:sparsenonsparse} we represent the time evolution of the velocities of this system. 
The free evolution of the system is represented in red. The evolution under the action of the sparse control $u^\circ$ is  in blue while in green the system under the action of a ``distributed'' control acting on all the four agents simultaneously with $\alpha_{1} = \cdots = \alpha_{4}=1/4$.
\\ 
The system reaches the consensus region within a time $t=3.076$ under the action of both the distributed and the sparse control.
\end{example}

\begin{example}\label{ex:20}
We consider a group of $20$ agents starting with positions on the unit circle and with velocities pointing in various directions. Namely, 
$$
x_{i}(0)=(\cos(i+\sqrt(2)),\cos(i+2\sqrt(2)) ) \mbox{ and } v_{i}(0)=(2 \sin(i \sqrt(3)-1), 2 \sin(i \sqrt(3)-2)).
$$  
The initial configuration is represented in Figure~\ref{fig:init}. We consider that the interaction potential, as in the Cucker--Smale system is of the form~\eqref{eq:CSpotential} with $K=\sigma=\beta=1$, that is
$$
a(x) = \frac{1}{1+x^{2}}.
$$
The sufficient condition for consensus \eqref{eq:CScondition} then reads
$$
\sqrt{V} \leq \frac{1}{\sqrt{2N}} \left(\frac{\pi}{2} - \arctan(\sqrt{2NX})\right).
$$
The system in free evolution does not tend to consensus, as showed in Figure~\ref{fig:uncontrolled}. After a time of $100$ the quantity $\sqrt{V(100)} \simeq 1.23$
while 
$
\gamma(X(100)) \simeq 0.10.$

On the other hand the componentwise sparse control steers the system to consensus in time $t=22.3$. Moreover the totally distributed control, acting on the whole group of $20$ agents, steers the system in a larger time, $t=27.6$. The time evolution of  $\sqrt{V}$ and of $\gamma(X)$ is represented in Figure~\ref{fig:20}. In Figure~\ref{fig:20zoom} the detail of the moment in which the two systems enter the consensus region.

\end{example}

\begin{figure}
\centering
\includegraphics[width=10cm]{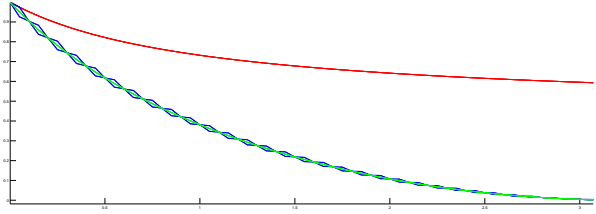}
\caption{The time evolution of the modulus of the velocities in the fully symmetric case of Example~\ref{ex:symm}. In  red the free evolution of the system, in blue the evolution under the action of a sparse control,
and in green the system under the action of a distributed control.}
\label{fig:sparsenonsparse}
\end{figure}

\begin{figure}
\centering
\includegraphics[width=6cm]{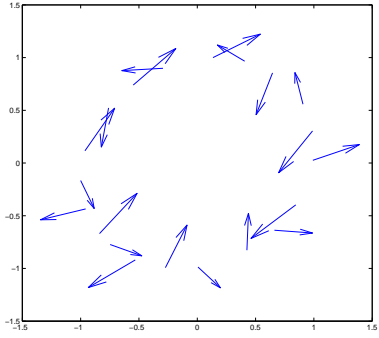}
\caption{The initial configuration of Example~\ref{ex:20}.}
\label{fig:init}
\end{figure}

\begin{figure}
\centering
\includegraphics[width=10cm]{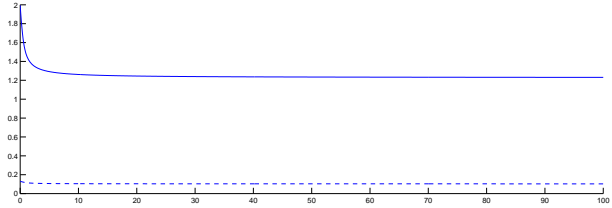}
\caption{The time evolution for $t\in [0,100]$ of $\sqrt{V(t)}$ (solid line) and of the quantity $\gamma(X(t))$ (dashed line). The system does not reach the consensus region.}
\label{fig:uncontrolled}
\end{figure}

\begin{figure}
\centering
\includegraphics[width=10cm]{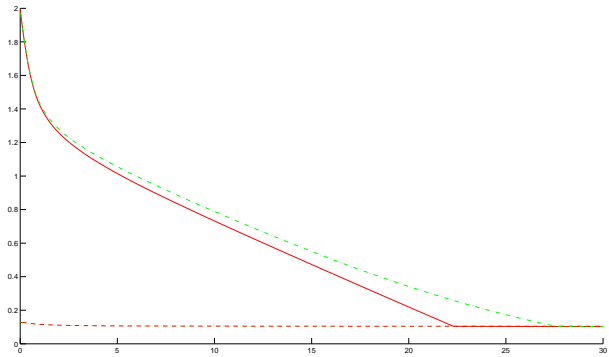}
\caption{Comparison between the actions of the componentwise sparse control and the totally distributed control. The time evolution for $t\in [0,30]$ of $\sqrt{V(t)}$ (solid line in the sparse case and dash-dot line in the distributed case) and of $\gamma(X(t))$ (dashed line in the sparse case and dotted line in the distributed case).}
\label{fig:20}
\end{figure}

\begin{figure}
\centering
\includegraphics[width=10cm]{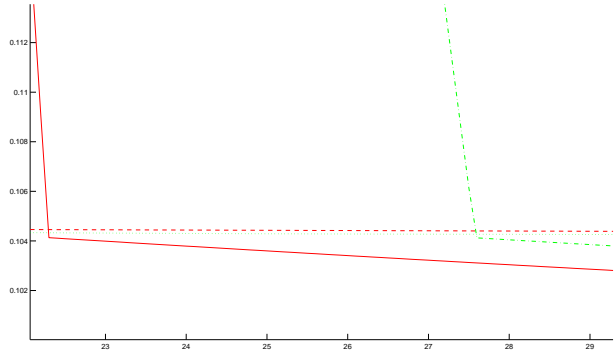}
\caption{Detail of the time evolution of $\sqrt{V(t)}$ and of $\gamma(X(t))$ under the action of the componentwise sparse control and the completely distributed control near the time in which the two systems enter the consensus region. The solid line represents the evolution of $\sqrt{V(t)}$ under the action of the componentwise sparse control and the dash-dot line the evolution of $\sqrt{V(t)}$ under the action of the distributed control. The dashed line represents the evolution of $\gamma(X(t))$ 
under the action of the componentwise sparse control and the  dotted line the evolution of $\gamma(X(t))$  under the action of the distributed control.}
\label{fig:20zoom}
\end{figure}

\subsection{Complexity of consensus}\label{sec:complexity}

The problem of determining minimal data rates for performing
control tasks has been considered for more than twenty years.
Performing control with limited data
rates incorporates ideas from both control and information
theory and it is an emerging area, see the survey Nair,
Fagnani, Zampieri and Evans \cite{nafazaev07}, and the references within the recent paper \cite{co12}.
Similarly, in {\it Information Based Complexity} \cite{ibc1,ibc2}, which is a branch of theoretical numerical analysis, one investigates which are the {\it minimal amount of algebraic operations} required by {\it any} algorithm in order 
perform accurate numerical approximations of functions, integrals, solutions of differential equations etc., given that the problem applies on a class of functions or on a class of solutions. 

We would like to translate
such concepts of {\it universal complexity} (universal because it refers to the best possible algorithm for the given problem over an entire class of functions) to our problem of optimizing the external intervention on the system in order to achieve consensus.

For that,  and for any vector $w \in \mathbb R^d$, let us denote $\operatorname{supp}(w):=\{ i \in \{1,\dots,d\}: u_i \neq 0 \}$ and $\# \operatorname{supp}(w)$ its cardinality.
Hence, we define the {\it minimal number of external interventions} as the sum of the actually activated components of the control $\# \operatorname{supp}(u(t_\ell))$ at each switching time $t_\ell$, which a policy maker should provide by using {\it any} feedback control strategy $u$ in order to steer the system to consensus at a given time $T$. Not being the switching times $t_0,t_1,\dots, t_\ell, \dots$ specified a priori, such a sum simply represents the amount of communication requested to the policy maker to activate and deactivate individual controls by informing the corresponding agents of the current mean consensus parameter $\bar v$ of the group. (Notice that here, differently from, e.g., \cite{nafazaev07}, we do not yet consider quantization of the information.)

More formally, given a suitable compact set $\mathcal K \subset \RN \times \RN$ of initial conditions, the $\ell_1^N-\ell_2^d$-norm control bound $M>0$, the set of corresponding admissible 
feedback controls $\mathscr U(M)$ with values in $B_{\ell_1^N-\ell_2^d}(M)$, the number of agents $N \in \mathbb N$,
and an arrival time $T>0$, we define the {\it consensus number} as
\begin{eqnarray*}
n&:=&n(N,\mathscr U(M),\mathcal K,T) \\
&=& \inf_{u \in \mathscr U(M)} \left \{ \sup_{(x_0,v_0) \in \mathcal K}\left \{ \sum_{\ell=0}^{k-1} \# \operatorname{supp}(u(t_\ell)): (x(T;u),v(T,u)) \mbox{ is in the consensus region }  \right \} \right \}.
\end{eqnarray*}
Although it seems still quite difficult to give a general lower bound to the consensus numbers, Theorem  \ref{prop:piecewise} actually allows us to provide at least upper bounds:  for 
$T_0 = T_0(M,N,x_0,v_0,a(\cdot)) = \frac{2N}{M}(\sqrt{V(0)} - \gamma(\bar X))$,
and $\tau_0=\tau_0(M,N,x_0,v_0,a(\cdot))$ as in Theorem \ref{prop:piecewise} and Remark \ref{rem:maxtime}, we have the following upper estimate 
\begin{equation}\label{eq:complx}
n(N,\mathscr U(M),\mathcal K,T) \leq \left \{ \begin{array}{ll}
\infty, & T< T_0\\
 \frac{\sup_{(x_0,v_0) \in \mathcal K}  T_0(M,N,x_0,v_0,a(\cdot))}{
\inf_{(x_0,v_0) \in \mathcal K}\tau_0(M,N,x_0,v_0,a(\cdot))},& T\geq T_0
\end{array} \right . .
\end{equation}
Depending on the particular choice of the rate of communication function $a(\cdot)$, such upper bounds can be actually computed, moreover, one can also quantify them over a class of  communication functions $a(\cdot)$ in a bounded set $\mathscr A \subset L^1(\mathbb R_+)$, simply by estimating the supremum.\\
The result of instantaneous optimality achieved in Proposition \ref{prop:sparseisbetter} suggests that the sampling strategy of Theorem \ref{prop:piecewise} is likely to be close to optimality in the sense that the upper bounds \eqref{eq:complx} should be
close to the actual consensus numbers. Clarifying this open issue will be the subject of further investigations which are beyond the scope of this paper.

\section{Sparse Controllability Near the Consensus Manifold}\label{sec:controll}
In  this section we address the problem of controllability near the consensus manifold. The stabilization results of Section~\ref{sec:stable} provide a constructive strategy to stabilize the multi-agent system~\eqref{control}: the system is first steered to the region of consensus, and then in free evolution reaches consensus in infinite time. Here we study the local controllability near consensus, and infer a global controllability result to consensus.

The following result states that, almost everywhere, local controllability near the consensus manifold is possible by acting on only one arbitrary component of a control, in other words whatever is the controlled agent it is possible to steer a group, sufficiently close to a consensus point, to any other desired close point.
Recall that the consensus manifold is $\RN\times\Vf$, where $\Vf$ is defined by \eqref{defVF}.

\begin{proposition}\label{prop:local}
For every $M>0$, for almost every $\tilde x \in \RN$ and for every $\tilde v \in \Vf$, for every time $T>0$, there exists a neighborhood $W$ of $(\tilde x,\tilde v)$ in $\RN\times\RN$ such that, for all points $(x_0,v_0)$ and $(x_1,v_1)$ of $W$, for every index $i \in \{1,\ldots,N\}$, there exists a componentwise and time sparse control $u$ satisfying the constraint \eqref{const_cont}, every component of which is zero except the $i^\textrm{th}$ (that is, $u_{j}(t) =0$ for every $j\neq i$ and every $t \in [0,T]$), steering the control system \eqref{control} from $(x_0,v_0)$ to $(x_1,v_1)$ in time $T$.
\end{proposition}

\begin{proof}
Without loss of generality we assume $i=1$, that is we consider the system~\eqref{control} with  a control acting only on the dynamics of $v_{1}$.
Given  $(\tilde x,\tilde v) \in \RN \times \Vf$ we  linearize the control system \eqref{control} at the consensus point $(\tilde x,\tilde v)$, and get $d$ decoupled systems on $\R^{N} \times \R^{N}$ 
\begin{equation*}
\left\{
\begin{split}
\dot x^{k} &= v^{k}\\
\dot v^{k} & = -L_{\tilde x} v^{k} + Bu\,,
\end{split}
\right.
\end{equation*}
for every $k=1, \ldots, d$  where
\begin{equation*}
B = \begin{pmatrix} 1 \\ 0\\ \vdots \\ 0 \end{pmatrix}.
\end{equation*}
To prove the local controllability result, we use the Kalman condition. It is sufficient to consider the  decoupled control sub-systems corresponding to each value of $k=1, \ldots, d$. Moreover the equations for $x^k$ do not affect the Kalman condition, the $x^{k}$ plays only the role of an integrator.
Therefore we reduce the investigation of the Kalman condition for a linear system on $\R^{N}$ of the form
\begin{equation}\label{eq:lapcontr}
\dot v = A v +B u, \mbox{ where } A = - L_{\bar x}. 
\end{equation}
Since $A$ is a Laplacian matrix then there exists an orthogonal matrix $P$ such that
$$
D:=P^{-1}AP = \begin{pmatrix}
0 & 0 & \cdots & 0 \\
0 & \lambda_2 & \ddots & \vdots \\
\vdots  &    \ddots      &  \ddots  &  0  \\
0  &    \cdots    &   0    &   \lambda_N        
\end{pmatrix}.
$$
Moreover  since $(1,\ldots,1) \in \ker A$,  we can choose all the coordinates of the first column of $P$ and thus the first line of $P^{-1}=P^T$ are equal to $1$. We denote the first column of $P^{-1}$ by
$$
B_1=\begin{pmatrix} 1 \\ \alpha_2 \\ \vdots \\ \alpha_N \end{pmatrix}.
$$
Notice that $B_1=P^{-1}B$.
Denoting the Kalman matrix of the couple $(A,B)$ by
$$
K(A,B) = (B,AB,\ldots,A^{N-1}B)
$$
one has
$$
K(P^{-1}AP,P^{-1}B) = P^{-1} K(A,B)
$$
and hence it suffices to investigate the Kalman condition on the couple of matrices $(D,B_1)$. Now, there holds
$$
K(D,B_1) = \begin{pmatrix}
1 & 0 &0  &\cdots & 0 \\
\alpha_2 & \lambda_2 \alpha_2 & \lambda_2^2\alpha_2 & \cdots & \lambda_2^{N-1}\alpha_2 \\
\vdots & \vdots &\vdots  &\vdots & \vdots \\
\alpha_N & \lambda_N \alpha_N & \lambda_N^2\alpha_N & \cdots & \lambda_N^{N-1}\alpha_N 
\end{pmatrix}.
$$
This matrix is invertible if and only if all eigenvalues $0, \lambda_2,\ldots,\lambda_N$ are pairwise distinct, and all coefficients $\alpha_2,\ldots,\alpha_N$ are nonzero.
It is clear that these conditions can be translated as algebraic conditions on the coefficients of the matrix $A$.

Hence, for almost every $\tilde x \in \RN$ and for every $\tilde v \in \Vf$, the Kalman condition holds at $(\tilde x,\tilde v)$. For such a point, this ensures that the linearized system at the equilibrium point $(\tilde x,\tilde v)$ is controllable (in any time $T$). Now, using a classical implicit function argument applied to the end-point mapping (see e.g. \cite{T_2005}), we infer the desired local controllability property in a neighborhood of $(\tilde x,\tilde v)$. By construction, the controls are componentwise and time sparse.
To prove the more precise statement of Remark \ref{remm8}, it suffices to invoke the chain of arguments developed in \cite[Lemma 2.1]{ST_TAC2010} and \cite[Section 2.1.3]{HT_SICON2011}, combining classical needle-like variations with a conic implicit function theorem, leading to the fact that the controls realizing local controllability can be chosen as a perturbation of the zero control with a finite number of needle-like variations.

\end{proof}

\begin{remark}
Actually the set of points $x \in \RN$ for which the condition is not satisfied can be expressed as an algebraic manifold in the variables $a(\Vert x_i-x_j\Vert)$. For example,
if $x$ is such that all mutual distances $\Vert x_i-x_j\Vert$ are equal, then it can be seen from the proof of this proposition that the Kalman condition does not hold, hence the linearized system around the corresponding consensus point is not controllable.
\end{remark}

\begin{remark}\label{remm8}
The controls realizing this local controllability can be even chosen to be piecewise constant, with a support union of a finite number of intervals.
\end{remark}

As a consequence of this local controllability result, we infer that we can steer the system from any consensus point to almost any other one by acting only on one agent. This is a partial but global controllability result, whose proof follows the strategy 
developed in \cite{coron-trelat04, coron-trelat06} for controlling heat and wave equations on steady-states. 

\begin{theorem}\label{thm:local}
For  every 
$(\tilde x_0,\tilde v_0) \in \RN \times \Vf$, for almost every
$(\tilde x_1,\tilde v_1) \in \RN \times \Vf$, for every $\delta >0$, and for every $i=1,\ldots,N$ 
there exist  $T>0$ and a control $u: [0,T] \to [0,\delta]^{d}$ steering the system from $(\bar x,\bar v)$ to $(\tilde x,\tilde v)$, with the property $u_{j}(t) =0$ for every $j\neq i$ and every $t \in [0,T]$.
\end{theorem}

\begin{proof}
Since the manifold of consensus points $\RN \times \Vf$ is connected, it follows that, for all consensus points $(\tilde x_0,\tilde v_0)$ and $(\tilde x_1,\tilde v_1)$, there exists a $C^1$ path of consensus points $(\tilde x_\tau,\tilde v_\tau)$ joining $(\tilde x_0,\tilde v_0)$ and $(\tilde x_1,\tilde v_1)$, and parametrized by $\tau\in[0,1]$.
Then the we apply iteratively the local controllability result of Proposition~\ref{prop:local}, on a series of neighborhoods covering this path of consensus points (his can be achieved by compactness). At the end, to reach exactly the final consensus point $(\tilde x_1,\tilde v_1)$, it is required that the linearized control system at $(\tilde x_1,\tilde v_1)$ be controllable, whence the ``almost every'' statement.
\end{proof}

Note that on the one hand the control $u$ can be of arbitrarily small amplitude, on the other hand the controllability time $T$ can be large.

Now, it follows from the results of the previous section that we can steer any initial condition $(x_0,v_0)\in\RN\times\RN$ to the consensus region defined by \eqref{eq:CScondition}, by means of a componentwise and time sparse control. Once the trajectory has entered this region, the system converges naturally (i.e., without any action: $u=0$) to some point of the consensus manifold $\RN\times\Vf$, in infinite time. This means that, for some time large enough, the trajectory enters the neighborhood of controllability whose existence is claimed in Proposition~\ref{prop:local}, and hence can be steered to the consensus manifold within finite time. Theorem~\ref{thm:local} ensures the existence of a control able move the system on the consensus manifold in order to reach almost any other desired consensus point. Hence we have obtained the following corollary.

\begin{corollary}\label{cor:lol}
For every $M>0$, for  every initial condition $(x_{0},v_{0}) \in \RN\times \RN$, for almost every $(x_1,v_1) \in \RN \times \Vf $, there exist $T >0 $ and a componentwise and time sparse control $u:[0,T] \to \RN$, satisfying \eqref{const_cont}, such that the corresponding solution starting at $(x_{0},v_{0})$ arrives at the consensus point $(x_1,v_1)$ within time $T$.
\end{corollary}

\section{Sparse Optimal Control of the Cucker--Smale Model}\label{sec:optimal}
In this section we investigate the sparsity properties of a  \textit{finite time optimal control} with respect to a cost functional involving the discrepancy of the state variables to consensus and a $\ell_1^N-\ell_2^d$-norm term of the control.

While the {\it greedy strategies} based on instantaneous feedback as presented in Section \ref{sec:stable} models the perhaps more realistic situation where the policy maker is not allowed to make future predictions, the optimal control problem presented in this section actually describes a model where the policy maker is allowed to see how the dynamics can develop. 
Although the results of this section do not lead systematically to sparsity, it is interesting to note that the {\it lacunarity of sparsity} of the optimal control is actually encoded in terms of the codimension of certain manifolds, which have actually null Lebesgue measure in the space of cotangent vectors. \\

We consider the optimal control problem of determining a trajectory solution of \eqref{control}, starting at $(x(0),v(0)) = (x_{0},v_{0}) \in \RN\times\RN$, and minimizing a cost functional which is a combination of the distance from consensus with the $\ell_{1}^N-\ell_2^d$-norm of the control (as in \cite{elra10,fora08}), under the control constraint \eqref{const_cont}. More precisely, the cost functional considered here is, for a given $\gamma>0$,
\begin{equation}\label{cost}
\int_0^T \bigg( \sum_{i=1}^N  \Big( v_i(t) - \frac{1}{N}\sum_{j=1}^N v_j(t) \Big)^2 + \gamma \sum_{i=1}^N\Vert u_i(t)\Vert \bigg) dt.
\end{equation}
Using classical results in optimal control theory (see for instance~\cite[Theorem~5.2.1]{BressanPiccoli} or \cite{Cesari,T_2005}), this optimal control problem has a unique optimal solution $(x(\cdot),v(\cdot))$, associated with a control $u$ on $[0,T]$, which is characterized as follows. According to the Pontryagin Minimum Principle (see \cite{Pontryagin}), there exist absolutely continuous functions $p_x(\cdot)$ and $p_v(\cdot)$ (called adjoint vectors), defined on $[0,T]$ and taking their values in $\RN$, satisfying the adjoint equations
\begin{equation}\label{eq:covector}
\left\{
\begin{split}
\dot p_{x_{i}}  & = \frac{1}{N}
 \sum_{j=1}^{N}  \frac{a(\|x_{j}-x_{i}\|)}{\|x_{j}-x_{i}\|} 
 \langle x_{j}-x_{i},v_{j} - v_{i} \rangle (p_{v_{j}} - p_{v_{i}})  , \\
\dot p_{v_{i}} &= -p_{x_{i}}  -  \frac{1}{N}  \sum_{j\neq i} a(\|x_{j} - x_{i}\|)(p_{v_{j}}- p_{v_{i}})- 2  v_{i} + \frac{2}{N}
\sum_{j=1}^{N} v_{j} ,
\end{split}
\right.
\end{equation}
almost everywhere on $[0,T]$, and $p_{x_i}(T)=p_{v_i}(T)=0$, for every $i=1,\ldots, N$. Moreover, for almost every $t\in[0,T]$ the optimal control $u(t)$ must minimize the quantity
\begin{equation}\label{eq:min}
\sum_{i=1}^N \langle p_{v_i}(t),w_i\rangle +  \gamma \sum_{i=1}^N\Vert w_i\Vert,
\end{equation}
over all possible $w=(w_1,\ldots,w_N)\in\RN$ satisfying $\sum_{i=1}^N\Vert w_i\Vert\leq M$.

In analogy with the analysis in Section~\ref{sec:stable} we identify five regions $\mathcal{O}_{1}, \mathcal{O}_{2}, \mathcal{O}_{3}, \mathcal{O}_{4}, \mathcal{O}_{5}$ covering the (cotangent) space $\RN\times\RN\times\RN\times\RN$:
\begin{itemize}
\item[$\mathcal{O}_{1}=$] $\{ (x,v,p_x,p_v)\ \vert\   \|p_{v_{i}}\| <  \gamma$ for every $i\in\{1,\ldots,N\}\}$,
\item[$\mathcal{O}_{2}=$] $\{ (x,v,p_x,p_v)\ \vert\  $ there exists a unique $i \in \{1,\ldots,N\}$ such that $  \|p_{v_{i}}\|  =  \gamma$  and $\|p_{v_{j}}\| <  \gamma$ for every $j \neq i \}$,
\item[$\mathcal{O}_{3}=$]
 $\{ (x,v,p_x,p_v)\ \vert\  $ there exists a unique $i \in \{1,\ldots,N\}$ such that $  \|p_{v_{i}}\| >  \gamma$ and $\|p_{v_{i}}\| >\|p_{v_{j}}\|$ for every $j \neq i \}$,
\item[$\mathcal{O}_{4}=$] 
$\{ (x,v,p_x,p_v)\ \vert\  $ there exist $k\geq 2$ and $i_{1},\ldots,i_{k} \in \{1,\ldots, N\}$  such that  $  \|p_{v_{i_{1}}}\|  = \|p_{v_{i_{2}}}\| = \cdots = \|p_{v_{i_{k}}}\| >  \gamma$  and $\|p_{v_{i_{1}}}\| > \|p_{v_{j}}\|$ for every $j \notin\{ i_{1},\ldots,i_{k}\}\}$,
\item[$\mathcal{O}_{5}=$] $\{ (x,v,p_x,p_v)\ \vert\  $ there exist $k\geq 2$ and  $i_{1},\ldots,i_{k} \in \{1,\ldots, N\}$ such that  $  \|p_{v_{i_{1}}}\|  = \|p_{v_{i_{2}}}\| = \cdots = \|p_{v_{i_{k}}}\|  =  \gamma$  and $\|p_{v_{j}}\| <  \gamma$ for every $j \notin\{ i_{1},\ldots,i_{k}\}\}$.
\end{itemize}
The subsets $\mathcal{O}_{1}$ and $\mathcal{O}_{3}$ are open, the submanifold $\mathcal{O}_{2}$ is closed (and of zero Lebesgue measure) and $\mathcal{O}_{1}\cup \mathcal{O}_{2}\cup \mathcal{O}_{3}$ is of full Lebesgue measure in $\RN\times \RN$. Moreover if an extremal $(x(\cdot),v(\cdot),p_x(\cdot),p_v(\cdot))$ solution of \eqref{control}-\eqref{eq:covector} is in $\mathcal{O}_{1}\cup \mathcal{O}_{3}$ along an open interval of time then the control is uniquely determined from \eqref{eq:min} and is componentwise sparse.
Indeed, if there exists an interval $I \subset [0,T]$ such that $(x(t),v(t),p_{x}(t),p_{v}(t)) \in \mathcal{O}_{1}$ for every $t \in I$,  then \eqref{eq:min} yields $u(t) = 0$ for almost every $t\in I$.
If $(x(t),v(t),p_{x}(t),p_{v}(t)) \in \mathcal{O}_{3}$ for every $t \in I$ then 
\eqref{eq:min} yields $u_{j}(t) =0$ for every $j \neq i$ and $u_{i}(t) = - M \frac{p_{v_{i}}(t)}{\|p_{v_{i}}(t)\|}$ for almost every $t\in I$. Finally, if $(x(t),v(t),p_{x}(t),p_{v}(t)) \in \mathcal{O}_{2}$ for every $t \in I$, then \eqref{eq:min} does not determine $u(t)$ in a unique way: it yields that $u_{j}(t) =0$ for every $j \neq i$ and $u_i(t)= - \alpha \frac{p_{v_{i}}(t)}{\|p_{v_{i}}(t)\|}$ with $0\leq \alpha\leq M$, for almost every $t\in I$. However $u$ is still componentwise sparse on $I$.
\\
The submanifolds $\mathcal{O}_{4}$ and $\mathcal{O}_{5}$ are of zero Lebesgue measure. When the extremal is in these regions, the control is not uniquely determined from \eqref{eq:min} and is not necessarily componentwise sparse. More precisely, if $(x(t),v(t),p_{x}(t),p_{v}(t)) \in \mathcal{O}_{4}\cup \mathcal{O}_{5}$ for every $t \in I$, then  \eqref{eq:min} is satisfied by every control of the form  $u_{i_{j}}(t) = - \alpha_{j}  \frac{ p_{v_{i_{j}}}(t)}{\|p_{v_{i_{j}}}(t)\|}$, $j=1,\ldots,k$, and $u_{l} =0$ for every $l \notin\{ i_{1},\ldots,i_{k}\}$, where the $\alpha_i$'s are nonnegative real numbers such that $0\leq \sum_{j=1}^k \alpha_{j} \leq M$ whenever $(x(t),v(t),p_{x}(t),p_{v}(t)) \in \mathcal{O}_{5}$, and such that $\sum_{j=1}^k \alpha_{j} = M$ whenever $(x(t),v(t),p_{x}(t),p_{v}(t)) \in \mathcal{O}_{4}$.
We have even the following more precise result.

\begin{proposition}\label{lem:O2O4}
The submanifolds $\mathcal{O}_{4}$ and $\mathcal{O}_{5}$ are stratified\footnote{in the sense of Whitney, see e.g. \cite{Goresky}.} manifolds of codimension larger than or equal to two. More precisely, $\mathcal{O}_{4}$ (resp., $\mathcal{O}_{5}$) is the union of submanifolds of codimension $2(k-1)$ (resp., $2k$), where $k$ is the index appearing in the definition of these subsets and it is as well the number of active components of the control at the same time.
\end{proposition}
\begin{proof}
Since the arguments are similar for $\mathcal{O}_{4}$ and $\mathcal{O}_{5}$,  we only treat in details the case of $\mathcal{O}_{4}$. Assume that $\|p_{v_{1}}(t)\|  = \|p_{v_{2}} (t)\| >  \gamma$, and that $\|p_{v_{j}}(t)\| < \|p_{v_{1}}(t)\|$ for every $j=3,\ldots,N$ and for every $t \in I$. 
Differentiating with respect to $t$ the equality  
\begin{equation}\label{eq:derivation0}
\|p_{v_{1}}(t)\|^{2}  = \|p_{v_{2}} (t)\|^{2},
\end{equation} 
we obtain
\begin{align}
 \langle p_{v_{2}}, p_{x_{2}}\rangle - \langle p_{v_{1}}, p_{x_{1}}\rangle &+ \frac{1}{N} \sum_{j=3}^{N}
 \langle p_{v_{j}}, 
  a(\|x_{j}-x_{2}\|)
 p_{v_{2}} - 
 a(\|x_{j}-x_{1}\|)
  p_{v_{1}} \rangle + 
 \nonumber \\
 &\qquad+ \frac{1}{N} \|p_{v_{1}}\|^{2} \sum_{j=3}^{N}\left( a(\|x_{j}-x_{1}\|) - a(\|x_{j}-x_{2}\|)\right) +\nonumber \\
&\qquad+ 2  ( \langle p_{v_{2}}, v_{2}\rangle - \langle p_{v_{1}}, v_{1}\rangle)
+ \frac{2}{N} \langle p_{v_{1}} - p_{v_{2}}, \sum_{j=1}^{N}v_{j} \rangle =0.\label{eq:derivation1}
\end{align}
These two relations are clearly independent in the cotangent space.
Since a vector must satisfy \eqref{eq:derivation0} and \eqref{eq:derivation1}, this means that the 
$\mathcal{O}_{4}$ is a submanifold of the cotangent space 
$\R^{4dN}$ of codimension $2$.
Assume now that $\|p_{v_{1}}(t)\|  = \|p_{v_{2}} (t)\|= \cdots = \|p_{v_{k}} (t)\|$, $\|p_{v_{1}}(t)\|  >  \gamma$, $\|p_{v_{j}}(t)\| < \|p_{v_{1}}(t)\| $ for $j = k+1, \ldots, N$, for every $t \in I$.
Then for every  pair $(p_{v_1},p_{v_j})$ $j=2, \ldots, k $
 we have a relation of the kind~\eqref{eq:derivation0} and a relation of the kind \eqref{eq:derivation1}.  Hence $\mathcal{O}_{4}$ has codimension  $2(k-1)$. It follows clearly that $\mathcal{O}_{4}$ is a stratified manifold, whose strata are submanifolds of codimension $2(k-1)$.
\end{proof}

It follows from these results that the componentwise sparsity features of the optimal control are coded in terms of the codimension of the above submanifolds.
By the way, note that, since $p_x(T)=p_v(T)=0$, there exists $\varepsilon>0$ such that $u(t)=0$ for every $t\in[T-\varepsilon,T]$. In other words, at the end of the interval of time the extremal $(x(\cdot),v(\cdot),p_x(\cdot),p_v(\cdot))$ is in $\mathcal{O}_1$.

It is an open question of knowing whether the extremal may lie on the submanifolds $\mathcal{O}_4$ or $\mathcal{O}_5$ along a nontrivial interval of time. What can be obviously said is that, for generic initial conditions $((x_{0},v_{0}),(p_x(0),p_v(0)))$, the optimal extremal does not stay in $\mathcal{O}_{4}\cup\mathcal{O}_{5}$ along an open interval of time; such a statement is however unmeaningful since the pair $(p_x(0),p_v(0))$ of initial adjoint vectors is not arbitrary and is determined through the shooting method by the final conditions $p_x(T)=p_v(T)=0$.

\section{Conclusions and Future Directions}\label{sec:ext}

In this paper we provided sparse feedback control strategies for inducing  alignment consensus in a group of agents driven by a Cucker--Smale type dynamics.  We clarified how these natural controls
stem from variational principles involving $\ell_1$-norm penalization terms.
 Not only we showed that sparse control is economical in terms of number of interactions of the external controller/policy maker with the group of agents, but we also proved its optimality with respect to a very large class of 
possible (also distributed) controls, in the sense of  instantaneously providing the  largest decrease of a Lyapunov functional measuring distance from consensus. This remarkable
property has never been highlighted in our studies. Building upon these preliminary results we have been able to clarify the global controllability of these systems, and we investigated also 
the sparsity of finite horizon optimal control subjected to $\ell_1$-norm penalization terms.
 \\

 Let us now give a  glimpse to some of the developments of this work.
Our approach extends to other model of social dynamics. Indeed, on one side
the specific form of the Cucker--Smale model \eqref{system} plays a significant role in the definition of the consensus region as motivated after Proposition \ref{prop:condition-for-consensus}. However, on the other side, it is its graph-Laplacian structure 
\begin{equation}\label{eq:procontr}
\left\{ \begin{split}
\dot x &= v \\
\dot v &= - L_{x} v,
\end{split}\right.
\end{equation} 
where $L_{x}$ is the Laplacian defined in Section \ref{sec:CSmodel}, which is responsible for the controllability of the system. In fact, the nonnegativity of $L_x$ with respect to the bilinear form $B(\cdot,\cdot)$ is a key ingredient which 
allows us in Proposition \ref{prop:stable}, Theorem \ref{thm:main}, and Theorem \ref{prop:piecewise} (here also the boundedness of the map $x \to L_x$ plays a role) to show convergence of the controlled system \eqref{control} towards the consensus region. In addition, for the proof of Theorem \ref{thm:main2} we just need
the continuity and the uniform boundedness of the map $x \to L_x$. Also the results of controllability, in particular the proof of Proposition \ref{prop:local} and its corollaries Theorem \ref{thm:local} and Corollary \ref{cor:lol}, depends exclusively
on the graph-Laplacian structure of the dynamics, see formula \eqref{eq:lapcontr}.  We conclude that the results mentioned  above can be easily adapted to dynamical systems of the type \eqref{eq:procontr}, where $L_x$ is a Laplacian matrix boundedly and continuously depending on the main state parameter $x$.\\
Let us however stress that our analysis has more far reaching potential, as it can address also situations which do not match the structure \eqref{eq:procontr}, such as the Cucker and Dong model of cohesion and avoidance \cite{CuckerDong11}, where the system has actually the form
\begin{equation}\label{eq:procontr2}
\left\{ \begin{split}
\dot x &= v \\
\dot v &= - (L_x^c - L_{x}^a) x,
\end{split}\right.
\end{equation} 
where $L_x^a$ and $L_x^c$ are graph-Laplacians associated to avoidance and cohesion forces respectively. In the recent work \cite{bofofrha13} the strategy proposed within the present paper has been generalized to the  Cucker and Dong mode, showing 
controllability, conditional to the initial conditions.\\

A number of further interesting research directions stems out from the present work, and we limit ourself in the following list to the mention of ongoing work in progress.
The latter include the following:
\begin{itemize}
\item[-] It is natural to address the mean-field limit of social dynamics models (see \cite{CCH13} for a recent survey
for uncontrolled systems) towards {\it sparse} control, connecting our work with the by now very broad literature of sparse
optimal control of partial differential equations \cite{caclku12,clku11,clku12,hestwa12,pive12,st09,wawa11}.
In particular we shall study infinite dimensional optimal control problems of a partial differential equation
of Vlasov-type, prescribing the dynamics of the probability distribution of interacting agents. A first step in this direction is achieved in the paper \cite{FS13}.
\item[-] In the non-flocking region the  Cucker-Dong system is expected 
to evolve into
a collection of clusters, each reaching consensus, see \cite{MT13} for a recent survey on heterophilious consensus. The problem of
controlling the number of clusters maybe interesting for a number of economic models.
\item[-] In socio-physics and opinion formation first order models (Krause type)
are often used. This would correspond to a dynamics with fixed positions for
the Cucker--Smale system. A natural question is how to extend our approach to such a case.
\item[-] Other investigations which are of interest to applications include:
sparse controls which are optimal from complexity point of view (see Section \ref{sec:complexity}),
observability of Cucker--Smale system, social dynamics systems with noise.
\end{itemize}

\section*{Acknowledgement}
Marco Caponigro acknowledges the support and the hospitality of the Department of Mathematics and the Center for Computational and Integrative Biology (CCIB) of Rutgers University during the preparation
of this work. 
Massimo Fornasier acknowledges the support of the ERC-Starting Grant ``High-Dimensional Sparse Optimal Control'' (HDSPCONTR - 306274). The authors acknowledge for the support the  NSFgrant \#1107444 (KI-Net).

\section{Appendix}\label{sec_appendix}

\subsection{Proof of Lemma~\ref{lem:app1}}\label{appendix_lem:app1}
For every $t \geq 0$, one has
\begin{align*}
\frac{d}{dt} \frac{1}{N} \sum_{i=1}^{N}\|v_{\perp_{i}}\|^{2} & = \frac{2}{N}  \sum_{i=1}^{N} \langle \dot v_{\perp_{i}}, v_{\perp_{i}} \rangle =  \frac{2}{N}  \sum_{i=1}^{N} \langle \dot v_{{i}}, v_{\perp_{i}} \rangle = \frac{2}{N^{2}} \sum_{i=1}^{N}  \sum_{j=1}^{N} a(\|x_{i} - x_{j}\|) \langle v_{j}-v_{i},v_{\perp_{i}} \rangle \\
&= 
\frac{1}{N^{2}}\left( \sum_{i=1}^{N}  \sum_{j=1}^{N} a(\|x_{i} - x_{j}\|) \langle v_{j}-v_{i},v_{\perp_{i}} \rangle
+
 \sum_{j=1}^{N}  \sum_{i=1}^{N} a(\|x_{j} - x_{i}\|) \langle v_{i}-v_{j},v_{\perp_{j}} \rangle
\right)\\
& = - \frac{1}{N^{2}} \sum_{i=1}^{N}  \sum_{j=1}^{N} a(\|x_{i} - x_{j}\|) \langle v_{i}-v_{j},v_{\perp_{i}} - v_{\perp_{j}} \rangle \\
& = - \frac{1}{N^{2}} \sum_{i=1}^{N}  \sum_{j=1}^{N} a(\|x_{i} - x_{j}\|) \|v_{i} - v_{j}\|^{2}.
\end{align*}
Now
 $$
\|x_{i} - x_{j}\| = \|x_{\perp_{i}} - x_{\perp_{j}}\|\leq \|x_{\perp_{i}} \| + \|x_{\perp_{j}} \| \leq \sqrt{2}\left(\sum_{i=1}^{N} \|x_{\perp_{i}}\|^{2}\right)^{\frac{1}{2}} = \sqrt{2NX}
$$
and since  $a$ is nonincreasing we have the statement.

\subsection{Proof of Proposition~\ref{prop:condition-for-consensus}}
\label{appendix_prop:condition-for-consensus}
We split the proof of Proposition~\ref{prop:condition-for-consensus} in several steps.
\begin{lemma}\label{lem:app2}
Assume that $V(0) \neq 0$, then for every $t\geq 0$
$$
\frac{d}{dt} \sqrt{V(t)}\leq -a \left(\sqrt{2NX(t)}\right) \sqrt{V(t)}.
$$
\end{lemma}
\begin{proof}
It is sufficient to remark that
$$
\frac{d}{dt} \sqrt{V(t)} = \frac{1}{2\sqrt{V(t)}} \frac{d}{dt} V(t)
$$
and apply Lemma~\ref{lem:app1}.
\end{proof}

\begin{lemma}\label{lem:app3} 
For every $t \geq 0$
$$
\frac{d}{dt} \sqrt{X(t)} \leq \sqrt{V(t)}
$$
\end{lemma}

\begin{proof}
Note that for the conservation of the mean consensus parameter $\dot x_{\perp_{i}} = v_{\perp_{i}}$. So
$$
\frac{1}{N} \frac{d}{dt} \sum_{i=1}^{N} \|x_{\perp_{i}}\|^{2} = \frac{2}{N} \sum_{i=1}^{N} \langle x_{\perp_{i}},v_{\perp_{i}} \rangle \leq  \frac{2}{N} \sum_{i=1}^{N} \|x_{\perp_{i}}\| \|v_{\perp_{i}}\|.
$$
The sum in the last term is the scalar product on $\R^{N}$ between the two vectors with components $\|x_{\perp_{i}}\|$
 and $\|v_{\perp_{i}}\|$ respectively. Applying once more the Cauchy-Schwarz inequality, on the one hand we have
 $$
 \frac{d}{dt} X(t) \leq  \frac{2}{N} \left(\sum_{i=1}^{N} \|x_{\perp_{i}}\|^{2}\right)^{\frac{1}{2}} \left( \sum_{i=1}^{N} \|v_{\perp_{i}}\|^{2}\right)^{\frac{1}{2}} = 2 \sqrt{X(t)} \sqrt{V(t)}.
 $$
 On the other hand
 $$
  \frac{d}{dt} X(t) =  \frac{d}{dt} \left (\sqrt{X(t)}  \sqrt{X(t)} \right )  = 2  \sqrt{X(t)} \frac{d}{dt} \sqrt{X(t)} .
 $$
\end{proof}

\begin{lemma}\label{lem:app4}
For every $t \geq 0$
\begin{equation}\label{eq:1182}
\sqrt{V(t)} + \int_{\sqrt{X(0)}}^{{\sqrt{X(t)}}} a(\sqrt{2N}r) dr \leq \sqrt{V(0)}.
\end{equation}
\end{lemma}

\begin{proof}
By Lemma~\ref{lem:app2} we have that
$$
\sqrt{V(t)} - \sqrt{V(0)} \leq - \int_{0}^{t}  a \left(\sqrt{2NX(s)}\right) \sqrt{V(s)}ds.
$$
Now set $r = \sqrt{X(s)} $. By Lemma~\ref{lem:app3} $- \sqrt{V(s)}ds \leq - dr$ and, therefore,
$$
\sqrt{V(t)} - \sqrt{V(0)} \leq - \int_{\sqrt{X(0)}}^{{\sqrt{X(t)}}} a\left(\sqrt{2N}r\right) dr .
$$
\end{proof}

Let us now end the proof of Proposition~\ref{prop:condition-for-consensus}.
If $V(0) =0$ then the system would be already in a consensus situation. Let us assume then that $V(0) >0$.
Since 
\begin{equation}\label{eq:3434}
0< \sqrt{V(0)} \leq \int_{\sqrt{X(0)}}^{\infty} a\left(\sqrt{2N}r\right) dr,
\end{equation}
then there exists  $\bar X > X(0)$ such that
\begin{equation}\label{eq:8793}
\sqrt{V(0)} =  \int_{\sqrt{X(0)}}^{{\sqrt{\bar X}}} a\left(\sqrt{2N}r\right) dr.
\end{equation}
Note that if in~\eqref{eq:3434} the equality holds then by taking the limit on both sides of~\eqref{eq:1182} we have that $\lim_{t\to\infty} V(t) =0.$
Otherwise we claim that $X(t) \leq \bar X$ for every $t \geq 0$ and we prove it by contradiction. Indeed if
there exists 
$\bar t$ such that $X(\bar t) > \bar X$ then by Lemma~\ref{lem:app4}
$$
\sqrt{V(0)} \geq \sqrt{V(\bar t)} + \int_{\sqrt{X(0)}}^{{\sqrt{X( \bar t)}}} a(\sqrt{2N}r) dr > 
 \int_{\sqrt{X(0)}}^{{\sqrt{\bar X}}} a(\sqrt{2N}r) dr = \sqrt{V(0)},
$$
that is a contradiction.
 
\subsection{On the invariance of  $\mathcal{C}_{1}$}\label{sec:invariance}

Here we prove the following technical lemma showing, in particular, that a trajectory originating in the region $\mathcal{C}_{1}$, as defined in Remark~\ref{rk:Uxv}, remains in that region. In other words the region 
$\mathcal{C}_{1}$ is positively invariant for the dynamics of~\eqref{system}.

\begin{lemma}\label{lem:invariance}
Let $(x(t),v(t))$ be a solution of~\eqref{system}. Then for every $t\geq 0$ we have
$$
\frac{d}{dt} \left(\max_{1\leq i \leq N}\|v_{\perp_{i}}(t)\| - \gamma(B(x(t),x(t)))\right) \leq 0. 
$$
\end{lemma}

\begin{proof}
Fix $t \geq 0$. Let $i\in \{1,\ldots,N\}$ be the index such that 
$$
\|v_{\perp_{i}}(t)\| \geq \|v_{\perp_{j}}(t)\| \quad \forall j=1,\ldots,N.
$$
Let us omit the dependence on $t$ for the sake of readability.  We have that 
\begin{align*}
\frac{d}{dt} \|v_{\perp_{i}}\|
&=  \frac{ \langle \dot v_{\perp_{i}}, v_{\perp_{i}}\rangle}{\|v_{\perp_{i}}\|} \\
& = \frac{ \langle \dot v_{i}, v_{\perp_{i}}\rangle}{\|v_{\perp_{i}}\|}\\
& = \frac{1}{N} \sum_{j=1}^N a(\Vert x_j-x_i\Vert) \frac{\langle v_j-v_i , v_{\perp_{i}}\rangle}{\|v_{\perp_{i}}\|}\\
& = \frac{1}{N} \sum_{j=1}^N a(\Vert x_j-x_i\Vert) \frac{\langle v_{\perp_{j}}-v_{\perp_{i}}, v_{\perp_{i}}\rangle}{\|v_{\perp_{i}}\|}\\
& = \frac{1}{N} \sum_{j=1}^N a(\Vert x_j-x_i\Vert) \left( \frac{\langle v_{\perp_{j}}, v_{\perp_{i}}\rangle}{\|v_{\perp_{i}}\|}
 - \|v_{\perp_{i}}\| \right)\\
 & \leq 
\frac{1}{N} a(\sqrt{2NX}) \sum_{j=1}^N \left( \frac{\langle v_{\perp_{j}}, v_{\perp_{i}}\rangle}{\|v_{\perp_{i}}\|}
 - \|v_{\perp_{i}}\| \right) \\
 & = - a(\sqrt{2NX})\|v_{\perp_{i}}\|,
 \end{align*}
since 
$$
\sum_{j=1}^{N} v_{\perp_{j}} =0,
$$
and 
$\|x_{k} - x_{j}\| \leq \sqrt{2NX}$.

On the other hand note that, by Lemma~\ref{lem:app3}, we have
$
\frac{d}{dt}\sqrt{X} \leq \|v_{\perp_{i}}\|.
$
In particular, since
$$
\frac{d}{dt} \gamma(X) = -a(\sqrt{2NX}) \frac{d}{dt}\sqrt{X},
$$
one has that
$$
\frac{d}{dt} (\|v_{\perp_{i}}\| - \gamma(B(x,x))) = 
\frac{d}{dt} \|v_{\perp_{i}}\| + a(\sqrt{2NX}) \frac{d}{dt}\sqrt{X} \leq 
- a(\sqrt{2NX})\|v_{\perp_{i}}\| + a(\sqrt{2NX})\|v_{\perp_{i}}\| =0,
$$
which concludes the proof.
\end{proof}

\subsection{A technical Lemma}\label{appendix:techlemm}

We state the following useful technical lemma, used in the proof of Theorem \ref{thm:main}.

\begin{lemma}\label{lem:integrate}
Let $(x(\cdot), v(\cdot))$ be a solution of~\eqref{eq:inclusion}. If there exist $\alpha >0$ and $T>0$ such that
\begin{equation}\label{eq:1235}
\frac{d}{dt} V(t) \leq - \alpha \sqrt{V(t)},
\end{equation}
for almost every $t \in[0,T]$, then 
\begin{equation}\label{eq:Kappa}
V(t) \leq \left(\sqrt{V(0)} - { \frac{\alpha}{2} } t\right)^{2},
\end{equation}
and
\begin{equation}\label{eq:Ics}
X(t) \leq 2 X(0) + \frac{2 N^2}{\alpha^{2}} V(0)^{2}\,.
\end{equation}
\end{lemma}

\begin{proof}
Let us remind that 
$$
X(t) = \frac{1}{2 N^2} \sum_{i,j=1}^N \|x_{i}(t) - x_{j}(t)\|^{2} \mbox{ and } V(t) = \frac{1}{2 N^2} \sum_{i,j=1}^N \|v_{i}(t) - v_{j}(t)\|^{2}.
$$
Integrating~\eqref{eq:1235} one has
$$
\int_{0}^{t} \frac{\dot{V}(s)}{\sqrt{V(s)}} ds \leq -\alpha t,
$$
and
$$
\sqrt{V(t)} - \sqrt{V(0)} =  \frac{1}{2}\int_{0}^{t} \frac{\dot{V}(s)}{\sqrt{V(s)}} ds \leq -\frac{\alpha}{2} t, 
$$
hence \eqref{eq:Kappa} follows.
For every $i,j \in \{1,\ldots, N\}$ we have
\begin{align*}
\|x_{i}(t) - x_{j}(t)\| & \leq \|x_{i}(0) - x_{j}(0)\| + \int_{0}^{t}\|v_{i}(s) - v_{j}(s)\| ds \\
& \leq  \|x_{i}(0) - x_{j}(0)\| + \sqrt{2}N \int_{0}^{t}\sqrt{V(s)} ds.
\end{align*}
Notice that here we used the estimate $\|v_{i}(s) - v_{j}(s)\|^2 \leq 2 N^2 \left( \frac{1}{2 N^2} \sum_{\ell,m=1}^N \|v_{\ell}(s) - v_{m}(s)\|^2  \right )=2 N^2 V(t)$.
Equation~\eqref{eq:1235}  implies also 
$$
\int_{0}^{t}\sqrt{V(s)} ds  \leq - \frac{1}{ \alpha}(V(t) - V({0})) < \frac{1}{ \alpha} V({0}).
$$
 Therefore, using the estimates as before, we have
\begin{eqnarray*}
X(t)= \frac{1}{2N^{2}} \sum_{i,j=1}^N \|x_{i}(t) - x_{j}(t)\|^{2} &\leq& \frac{1}{2N^{2}} \sum_{i,j=1}^N2 \left( \|x_{i}(0) - x_{j}(0)\|^{2} +
 \left ( \int_0^t \|v_{i}(s) - v_{j}(s)\| ds \right )^2 \right)\\
&\leq & \frac{1}{2N^{2}} \sum_{i,j=1}^N \left ( 2 \|x_{i}(0) - x_{j}(0)\|^{2} + 4 N^2 \left(  \int_0^t  \sqrt{V(s)} ds \right )^2 \right)\\
&\leq & 2 \left( \frac{1}{2N^{2}} \sum_{i,j=1}^N  \|x_{i}(0) - x_{j}(0)\|^{2} \right)  + 2 \left ( \sum_{i,j=1}^N \frac{V(0)^2}{\alpha^2} \right) \\
&=& 2 X(0) +\frac{2 N^2}{\alpha^2} V(0)^2.
\end{eqnarray*}
\end{proof}

\bibliographystyle{abbrv}
\bibliography{biblioflock}

\end{document}